%% file: main.tex
\documentclass[11pt,a4paper,dvipsnames,x11names]{amsart}

\input{pre}

\title[Koszul duality]{Contravariant Koszul duality between non-positive and positive dg algebras
}
\author{Riku Fushimi}
\date{\today}

\newcommand{\Addresses}{{% additional braces for segregating \footnotesize
  \bigskip
  \footnotesize

  R. Fushimi, \textsc{Department of Mathematics, Nagoya University, Chikusa-ku, Nagoya 464-8602, Japan}\par\nopagebreak
  \textit{E-mail address}: \texttt{fushimi.riku.h9@s.mail.nagoya-u.ac.jp}

}}

\begin{document}

\begin{abstract}
We study contravariant Koszul duality between locally finite non-positive and positive dg algebras. We characterize a finiteness condition on the positive side, called pvd-finiteness, in terms of Hom-finiteness of the perfectly valued derived category, local finiteness of the Koszul dual, and the existence of an injective cogenerator in the canonical heart. Under this condition, Koszul duality exchanges perfect and perfectly valued derived categories, and it gives a bijection between quasi-isomorphism classes of locally finite non-positive and pvd-finite positive dg algebras. We also show that the derived category of a pvd-finite positive dg algebra is co-approximable in the sense of Rahul, and that the perfect and perfectly valued derived categories of such an algebra determine each other through finite cohomological functors. As applications, we give a Morita-type characterization of these triangulated categories and show that a bounded heart of such a category is a length category if and only if it is functorially finite.
\end{abstract}

\maketitle
\setcounter{tocdepth}{1}
\tableofcontents
%%%%%%%%%%%%%%%%%%%%%%%%%%%%%%%%%%%%%%%%%%%%%%%%%%%%%%%%%%%%%%%%%%%%%%%%%%%%%%%%%%%%%%%%%%%%%%%%%%%%%%%%%%%%%%%%%%%%%%%%%%%%%%%%%%%%%%%%%%%
\section{Introduction}

Koszul duality goes back to Priddy \cite{P70} and Beilinson--Ginzburg--Soergel \cite{BGS96} in the setting of Koszul algebras. For dg algebras it relates non-positive dg algebras to positive ones, but this relation is asymmetric with respect to local finiteness: the Koszul dual of a locally finite non-positive dg algebra is always locally finite, whereas the Koszul dual of a locally finite positive dg algebra need not be. In this paper we determine exactly when the Koszul dual of a locally finite positive dg algebra is locally finite. Under this condition, Koszul duality becomes a contravariant equivalence between the two sides, at the level of perfect, perfectly valued and one-sided finite derived categories, without any augmentation or connectedness assumption. We then characterize the resulting triangulated categories intrinsically, show that the perfect and perfectly valued derived categories determine each other through finite cohomological functors, and characterize length hearts among all bounded hearts by approximation properties.

Throughout, dg algebras are cohomologically graded over a field $k$. A dg algebra is \emph{non-positive} if $H^{>0}=0$, and \emph{positive} if $H^{<0}=0$ and $H^0$ is semisimple. We call it \emph{locally finite} if each cohomology group is finite-dimensional over $k$. For a dg algebra $C$, write $\per(C)$ and $\pvd(C)$ for its perfect and perfectly valued derived categories, and $\Dmfd(C)$ (resp.\ $\Dpfd(C)$) for the full subcategory of $\D(C)$ of objects whose cohomology is bounded above (resp.\ below) and finite-dimensional in each degree. For a locally finite dg algebra $C$ of either kind, the \emph{distinguished module} $S_C$ is $\top H^0(C)$ if $C$ is non-positive, and the dg $C$-module with cohomology $H^0(C)$ if $C$ is positive \cite{KN13}. Set
\[
	C^!:=\REnd_C(S_C),
	\qquad
	\Phi_C:=\RHom_C(-,S_C)\colon\D(C)\longrightarrow\D(C^{!,o})^o
\]
for the Koszul dual and the corresponding contravariant functor, where $C^{!,o}:=(C^!)^o$; similarly, $C^{o,!}:=(C^o)^!$.

The asymmetry already occurs for finite-dimensional dg algebras. The dg algebra $A=k[y]/(y^2)$ with $|y|=1$ and $d=0$ is proper and positive, but $A^!\simeq k[\![x]\!]$ with $|x|=0$, so $H^0(A^!)$ is infinite-dimensional (\cref{exa:non-pvd-finite}).

For a locally finite positive dg algebra $A$, Keller--Nicol\'as showed that $\per(A)$ carries a canonical bounded $t$-structure whose heart, the extension closure $\H_A$ of $\add A$, is a length category whose simple objects are, up to isomorphism, the indecomposable summands of $A$ \cite[Theorem~8.1]{KN13}. Our first result shows that local finiteness of $A^!$ is detected by this heart.

\begin{main}[\cref{thm:main}]\label{intro:pvd}
	Let $A$ be a locally finite positive dg algebra. The following conditions are equivalent:
	\begin{enumerate}[label=(\roman*),leftmargin=2.4em]
		\item the canonical heart $\H_A$ has an injective cogenerator;
		\item the Koszul dual $A^!$ is locally finite;
		\item the category $\pvd(A)$ is Hom-finite;
		\item $\End_{\D(A)}(S_A)$ is finite-dimensional over $k$.
	\end{enumerate}
\end{main}

We call a locally finite positive dg algebra satisfying these conditions \emph{pvd-finite}. Since $\End_{\D(A)}(S_A)=H^0(A^!)$, condition (iv) says that local finiteness of $A^!$ is controlled by its degree-zero part. We also give concrete criteria: a directed pattern of degree-one extensions among the summands of $A$ implies pvd-finiteness (\cref{prop:directed-pvd}), and for a minimal tensor model pvd-finiteness amounts to having finitely many generators in degree one (\cref{prop:augmented-minimal-model}).

\begin{main}[\cref{thm:KD}]\label{intro:KD}
\textup{(1)} Let $B$ be a locally finite non-positive dg algebra and put $A=B^{!,o}$. Then $A$ is a pvd-finite positive dg algebra, $\Phi_B(B)\simeq S_A$, $\Phi_B(S_B)\simeq A$, and $\Phi_B$ induces triangle equivalences
\[
\begin{tikzcd}[column sep=4.6em,row sep=0.55em]
\per(B) \arrow[d,phantom,"\rotatebox{-90}{$\subseteq$}"{description}] \arrow[r,"\sim"]
    & \pvd(A)^o \arrow[d,phantom,"\rotatebox{-90}{$\subseteq$}"{description}] \\
\Dmfd(B) \arrow[r,"\sim"]
    & \Dpfd(A)^o \\
\pvd(B) \arrow[u,phantom,"\rotatebox{90}{$\subseteq$}"{description}] \arrow[r,"\sim"]
    & \per(A)^o \arrow[u,phantom,"\rotatebox{90}{$\subseteq$}"{description}].
\end{tikzcd}
\]
\textup{(2)} Let $A$ be a pvd-finite positive dg algebra and put $B=A^{!,o}$. Then $B$ is a locally finite non-positive dg algebra, $\Phi_A(A)\simeq S_B$, $\Phi_A(S_A)\simeq B$, and $\Phi_A$ induces triangle equivalences
\[
\begin{tikzcd}[column sep=4.6em,row sep=0.55em]
\per(A) \arrow[d,phantom,"\rotatebox{-90}{$\subseteq$}"{description}] \arrow[r,"\sim"]
    & \pvd(B)^o \arrow[d,phantom,"\rotatebox{-90}{$\subseteq$}"{description}] \\
\Dpfd(A) \arrow[r,"\sim"]
    & \Dmfd(B)^o \\
\pvd(A) \arrow[u,phantom,"\rotatebox{90}{$\subseteq$}"{description}] \arrow[r,"\sim"]
    & \per(B)^o \arrow[u,phantom,"\rotatebox{90}{$\subseteq$}"{description}].
\end{tikzcd}
\]
\end{main}

Here the locally finite non-positive dg algebra $B$ is arbitrary: we impose no connectedness or Koszulity assumption, and we do not choose an augmentation; its role is played by the distinguished module.

Contravariant functors of this form have appeared before. Keller developed Koszul duality for augmented dg categories \cite[Section~10]{Ke94}, and Keller--Nicol\'as proved that $\RHom_A(-,S_A)$ is fully faithful on $\per(A)$ for positive dg algebras under degreewise finiteness assumptions \cite[Proposition~6.4]{KN13}. Koszul dualities between derived categories were obtained by He--Wu for Koszul dg algebras \cite{HW08,HW10}, by Lu--Palmieri--Wu--Zhang for Adams-connected $A_\infty$-algebras \cite{LP07}, and by Avramov, in contravariant form, under connectedness hypotheses \cite[Theorem~9.1]{A13}. Su--Yang used Koszul duality to construct silting objects from simple-minded collections \cite{SY19}.

Double Koszul duality has been studied in augmented settings, in connection with derived completion \cite{DGI06,LP07,Boo22}. In our setting, the natural morphism $C\to C^{!!}=(C^!)^!$ is a quasi-isomorphism for every locally finite non-positive and every pvd-finite positive dg algebra $C$ (\cref{thm:Koszul-double-dual}). Consequently, Koszul duality gives a correspondence between dg algebras up to quasi-isomorphism.

\begin{main}[\cref{cor:KD-on-alg}]\label{intro:KD-on-alg}
Koszul duality induces a bijection
\[
\left\{
\begin{array}{c}
\text{locally finite non-positive}\\
\text{dg algebras}
\end{array}
\right\}\big/\simeq
\quad\longleftrightarrow\quad
\left\{
\begin{array}{c}
\text{pvd-finite positive}\\
\text{dg algebras}
\end{array}
\right\}\big/\simeq,
\]
where $\simeq$ denotes quasi-isomorphism. The correspondence in either direction is given by $C\mapsto C^{!,o}$.
\end{main}

The duality also yields an intrinsic description of the triangulated categories involved. Let $\T$ be a Hom-finite algebraic triangulated category with a \emph{length heart} $\H$, that is, a heart of a bounded $t$-structure on $\T$ that is a length category with only finitely many isomorphism classes of simple objects. Then $\T\simeq\per(A)$ for the locally finite positive dg algebra $A=\REnd(L)$, where $L$ is the direct sum of representatives of the isomorphism classes of simple objects of $\H$ \cite{LC07}, \cite[Theorem~3.8(b)]{Ke06}. We show that $\T$ is moreover of the form $\pvd(B)^o$ if and only if $\H$ has an injective cogenerator (\cref{thm:morita-characterization,thm:length-has-injcogen}).

\begin{main}[\cref{thm:morita-characterization}]\label{intro:recognition}
	Let $\T$ be a Hom-finite algebraic triangulated category. The following conditions are equivalent:
	\begin{enumerate}[label=(\roman*),leftmargin=2.4em]
		\item $\T$ has a length heart with an injective cogenerator;
		\item $\T\simeq\per(A)$ for some pvd-finite positive dg algebra $A$;
		\item $\T\simeq\pvd(B)^o$ for some locally finite non-positive dg algebra $B$.
	\end{enumerate}
\end{main}

This class of triangulated categories is closed under passage to opposite categories (\cref{rmk:morita-opposite}), so the injective cogenerator in \textup{(i)} may be replaced by a projective generator.

For an essentially small $k$-linear triangulated category $\mathcal C$, a cohomological functor $F\colon\mathcal C^o\to\mod k$ is \emph{finite} if $F(\Sigma^iC)=0$ for $|i|\gg0$ for every $C\in\mathcal C$; we write $\operatorname{Coh}_{\mathrm f}(\mathcal C)$ for the category of finite cohomological functors.

\begin{main}[\cref{thm:restricted-yoneda}]\label{intro:representability}
Let $A$ be a pvd-finite positive dg algebra. There are equivalences
\begin{enumerate}[label=(\roman*),leftmargin=2.4em]
\item
\[
\pvd(A)\xrightarrow{\ \sim\ }\operatorname{Coh}_{\mathrm f}(\per(A)),
\qquad X\longmapsto\Hom_{\D(A)}(-,X)|_{\per(A)};
\]
\item
\[
\per(A)^o\xrightarrow{\ \sim\ }\operatorname{Coh}_{\mathrm f}(\pvd(A)^o),
\qquad P\longmapsto\Hom_{\D(A)}(P,-)|_{\pvd(A)}.
\]
\end{enumerate}
\end{main}

The two equivalences come from different representability theorems. The first combines the co-approximability of $\D(A)$ (\cref{thm:co-approximable}) with Rahul's representability theorem \cite{Rah25}. The second follows from \cref{intro:KD} and Neeman's representability theorem \cite[Theorem~1.4(i)]{Nee26}, applied to the non-positive dg algebra $A^{!,o}$. Thus each of $\per(A)$ and $\pvd(A)$ is recovered, as a $k$-linear category, from the triangulated structure of the other.

Finally, we characterize length hearts among all bounded hearts. For a locally finite non-positive dg algebra $B$, the length hearts of $\pvd(B)$ correspond to silting objects of $\per(B)$ and are functorially finite in $\pvd(B)$ \cite[Theorem~3.7 and Corollary~2.36]{F23}; for finite-dimensional algebras, see \cite{KY14} and \cite[Corollary~2.12]{SPP22}. Since every category of \cref{intro:recognition} is of the form $\pvd(B)^o$, length hearts are functorially finite in all of them. Conversely, a bounded heart of such a category that is covariantly or contravariantly finite, or has enough projectives or enough injectives, is already a length heart.

\begin{main}[\cref{thm:ff-heart-characterization}]\label{intro:hearts}
Let $\T$ be a Hom-finite algebraic triangulated category satisfying the equivalent conditions of \cref{intro:recognition}, and let $\H$ be any bounded heart of $\T$. The following conditions are equivalent:
\begin{enumerate}[label=(\roman*),leftmargin=2.4em]
\item $\H$ is a length heart;
\item $\H$ has a projective generator or an injective cogenerator;
\item $\H$ has enough projectives or enough injectives;
\item $\H$ is covariantly finite or contravariantly finite in $\T$;
\item $\H$ is functorially finite in $\T$.
\end{enumerate}
When these conditions hold, $\H\simeq\mod\Lambda$ for some finite-dimensional $k$-algebra $\Lambda$.
\end{main}

Coelho Sim\~oes, Pauksztello and Ploog also related approximation properties of bounded hearts to the existence of enough projective or injective objects under a saturation hypothesis \cite[Corollary~2.8]{SPP22}. Our hypothesis comes from Koszul duality instead and does not require saturation. Some hypothesis on the ambient category is necessary: the canonical length heart of $\per(k[y]/(y^2))$ with $|y|=1$ is not contravariantly finite (\cref{exa:non-pvd-finite,thm:pvd-finite-heart}). As a further application, we recover Smal\o{}'s symmetry for functorially finite torsion pairs in module categories \cite{S84} using HRS tilting \cite{H96} (\cref{cor:smalo-module-symmetry}); \cref{thm:smalo-type} is a dg analogue of this symmetry for the canonical $t$-structure.

\subsection*{Organization}

\Cref{section:tri} recalls approximations, $t$-structures, weight structures, simple-minded collections and silting objects, together with basic notions for dg algebras. \Cref{section:dgalg} studies the canonical structures on the non-positive and positive sides and develops the finiteness criteria needed for the duality. In \cref{section:koszul-duality} we establish the contravariant equivalences, the characterizations of pvd-finiteness, the Morita characterization, the description of length hearts as module categories, the silting--$t$-structure (ST) correspondence, double duality and an adjunction. Finally, \cref{section:ff} treats co-approximability, representability and the functorial finiteness of bounded hearts, including the application to Smal\o{}'s symmetry.

\section{Preliminaries}\label{section:tri}

Throughout, $k$ is a field. All categories and functors are $k$-linear, all dg algebras are cohomologically graded, and all dg modules are right modules. All subcategories are assumed to be full and closed under isomorphisms.

\subsection{Triangulated categories and approximations}

Let $\T$ be a triangulated category with suspension functor $\Sigma$. We write
\[
	X\longrightarrow Y\longrightarrow Z\longrightarrow
\]
for a distinguished triangle, suppressing the final morphism $Z\to\Sigma X$ unless it is needed. For classes of objects $\X,\Y\subseteq\T$, put
\[
	\X\ast\Y:=\{Z\in\T\mid \text{there is a triangle }X\to Z\to Y\to\text{ with }X\in\X,\ Y\in\Y\}.
\]
We write $\add\X$ (resp.\ $\Add\X$) for the closure of $\X$ under finite (resp.\ arbitrary) direct sums and direct summands, $\thick(\X)$ for the smallest thick subcategory containing $\X$, and $\Filt(\X)$ for the smallest subcategory containing $\X\cup\{0\}$ and closed under extensions. For a single object $X$, we write $\add X:=\add\{X\}$ and $\thick X:=\thick\{X\}$. For a subcategory $\X\subseteq\T$, put
\begin{align*}
	\X^\perp&:=\{Y\in\T\mid\Hom_\T(X,Y)=0\text{ for all }X\in\X\},\\
	{}^\perp\X&:=\{Y\in\T\mid\Hom_\T(Y,X)=0\text{ for all }X\in\X\}.
\end{align*}

A triangulated category $\T$ is \emph{Hom-finite} if $\dim_k\Hom_\T(X,Y)<\infty$ for all $X,Y\in\T$. A Hom-finite idempotent-complete additive $k$-linear category is Krull--Schmidt \cite{Kr15}; in particular, the Hom-finite thick subcategories of derived categories occurring below are Krull--Schmidt. We call $\T$ \emph{algebraic} if it admits a dg enhancement. An object $G\in\T$ is a \emph{thick generator} if $\thick G=\T$. By Keller's dg Morita theorem \cite[Theorem~3.8(b)]{Ke06} (see also \cite{Ke94}), if $\T$ is algebraic, idempotent complete, and has a thick generator $G$, then $\T\simeq\per(A)$, where $A=\REnd(G)$ is the dg endomorphism algebra of $G$, with $G$ corresponding to $A$.

\begin{dfn}\label{dfn:approximation}
	Let $\C$ be an additive category and $\X\subseteq\C$ a subcategory. A morphism $f\colon X\to Y$ with $X\in\X$ is a \emph{right $\X$-approximation} of $Y$ if $\Hom_\C(X',f)$ is surjective for every $X'\in\X$; it is \emph{minimal} if every endomorphism $g$ of $X$ with $fg=f$ is an automorphism. Left approximations are defined dually. A subcategory $\X$ is \emph{contravariantly finite} (resp.\ \emph{covariantly finite}) in $\C$ if every object of $\C$ admits a right (resp.\ left) $\X$-approximation, and \emph{functorially finite} if both conditions hold. These notions go back to \cite{AS80}.
\end{dfn}

An object $P$ of an abelian category $\A$ is a \emph{projective generator} if it is projective and every object of $\A$ is a quotient of an object of $\add P$; \emph{injective cogenerators} are defined dually. The category $\A$ has \emph{enough projectives} (resp.\ \emph{enough injectives}) if every object is a quotient of a projective object (resp.\ a subobject of an injective object).

\begin{fct}[{Triangulated Wakamatsu lemma, \cite[Lemma~2.1]{J09}}]\label{fct:wakamatsu}
Let $\T$ be a Hom-finite Krull--Schmidt triangulated category and let $\X\subseteq\T$ be closed under extensions and direct summands. If $f\colon X\to Y$ is a minimal right $\X$-approximation and
\[
	X\xto{f}Y\longrightarrow C\longrightarrow
\]
is a triangle, then $C\in\X^\perp$.
\end{fct}

\subsection{\texorpdfstring{$t$}{t}-structures, weight structures and silting}

We use upper indices for $t$-structures and lower indices for weight structures. For a subcategory $\U^{\le0}\subseteq\T$ (resp.\ $\U^{\ge0}\subseteq\T$) and $n\in\ZZ$, put
\[
	\U^{\le n}=\U^{<n+1}:=\Sigma^{-n}\U^{\le0}
	\qquad
	\bigl(\text{resp.\ }\U^{\ge n}=\U^{>n-1}:=\Sigma^{-n}\U^{\ge0}\bigr),
\]
and similarly $\U_{\le n}=\U_{<n+1}:=\Sigma^{-n}\U_{\le0}$ and $\U_{\ge n}=\U_{>n-1}:=\Sigma^{-n}\U_{\ge0}$.

\begin{dfn}
	A pair $(\T^{\le0},\T^{\ge0})$ is a \emph{$t$-structure} on $\T$ \cite{BBD82} if
	\begin{itemize}
		\item $\T^{<0}\subseteq\T^{\le0}$ and $\T^{>0}\subseteq\T^{\ge0}$;
		\item $\Hom_\T(\T^{\le0},\T^{>0})=0$;
		\item $\T=\T^{\le0}\ast\T^{>0}$.
	\end{itemize}
	Its \emph{heart} is $\T^0:=\T^{\le0}\cap\T^{\ge0}$. The $t$-structure is \emph{bounded} if
	\[
		\bigcup_{n\in\ZZ}\T^{\le n}=\T=\bigcup_{n\in\ZZ}\T^{\ge n}.
	\]
\end{dfn}

We write $\tau^{\le n}$ and $\tau^{\ge n}$ for the truncation functors; when several $t$-structures occur, we add a subscript indicating the $t$-structure. A subcategory $\H\subseteq\T$ is a \emph{bounded heart} if it is the heart of a bounded $t$-structure. A bounded heart is a \emph{length heart} if it is a length category with only finitely many isomorphism classes of simple objects; if $S_1,\ldots,S_r$ represent them, then $\H=\Filt(\{S_1,\ldots,S_r\})$. A bounded $t$-structure is \emph{algebraic} if its heart is a length heart.

\begin{rmk}\label{rmk:contained-bounded-hearts}
A bounded $t$-structure is determined by its heart $\H$: its aisle and co-aisle are $\bigcup_{n\ge0}\Sigma^n\H\ast\cdots\ast\Sigma\H\ast\H$ and $\bigcup_{n\ge0}\H\ast\Sigma^{-1}\H\ast\cdots\ast\Sigma^{-n}\H$.
\end{rmk}

\begin{lem}\label{lem:heart-approx}
Let $\T$ be a Hom-finite Krull--Schmidt triangulated category and let $\H$ be the heart of a $t$-structure $(\T^{\le0},\T^{\ge0})$ on $\T$.
\begin{itemize}
	\item[(1)] If $\H$ is contravariantly finite in $\T$, then $\H$ has enough injectives.
	\item[(2)] If $\T^{\le0}$ is covariantly finite in $\T$, then so is $\H$.
\end{itemize}
The dual statements hold for covariantly finite hearts and contravariantly finite co-aisles.
\end{lem}
\begin{proof}
(1) We recall the argument from the proof of \cite[Theorem~2.4]{SPP22}. For $X\in\H$, choose a minimal right $\H$-approximation $N\to\Sigma X$ and complete it to a triangle $X\to I\to N\to\Sigma X$. Then $I\in\H$, since $\H$ is extension-closed, and \cref{fct:wakamatsu} gives $\Ext^1_{\H}(\H,I)=\Hom_\T(\H,\Sigma I)=0$. Thus $X\hookrightarrow I$ with $I$ injective in $\H$.

(2) For $X\in\T$, choose a left $\T^{\le0}$-approximation $X\to U$. Since $\Hom_\T(\tau^{\le-1}U,\H)=0$, every morphism from $U$ to an object of $\H$ factors through $U\to\tau^{\ge0}U\in\H$. Hence the composite $X\to U\to\tau^{\ge0}U$ is a left $\H$-approximation.
\end{proof}

\begin{dfn}\label{dfn:SMC}
	A finite collection $\mathcal S=\{S_1,\ldots,S_r\}$ of objects of $\T$ is a \emph{simple-minded collection} \cite{KY14} if
	\begin{itemize}
		\item $\Hom_\T(S_i,\Sigma^mS_j)=0$ for all $m<0$ and all $i,j$;
		\item $\Hom_\T(S_i,S_j)=0$ for $i\neq j$, and $\End_\T(S_i)$ is a division algebra for every $i$;
		\item $\thick(\mathcal S)=\T$.
	\end{itemize}
\end{dfn}

\begin{dfn}
	A pair $(\T_{\ge0},\T_{\le0})$ is a \emph{weight structure} \cite{Bon10} on $\T$ if
	\begin{itemize}
		\item $\T_{\ge0}$ and $\T_{\le0}$ are closed under direct summands;
		\item $\T_{>0}\subseteq\T_{\ge0}$ and $\T_{<0}\subseteq\T_{\le0}$;
		\item $\Hom_\T(\T_{\ge0},\T_{<0})=0$;
		\item $\T=\T_{\ge0}\ast\T_{<0}$.
	\end{itemize}
	Its \emph{coheart} is $\T_0:=\T_{\ge0}\cap\T_{\le0}$. The weight structure is \emph{bounded} if
	\[
		\bigcup_{n\in\ZZ}\T_{\ge n}=\T=\bigcup_{n\in\ZZ}\T_{\le n}.
	\]
\end{dfn}

Our indexing of weight structures is cohomological: for a positive dg algebra $A$, the canonical weight structure of \cite[Corollary~5.1]{KN13} has $\D(A)_{\ge0}=\{X\mid H^{<0}(X)=0\}$ (\cref{prop:canonical-weight-positive}).

\begin{dfn}
	An object $M\in\T$ is a \emph{silting object} \cite{KV88,AI12} if
	\begin{itemize}
		\item $\thick M=\T$;
		\item $\Hom_\T(M,\Sigma^iM)=0$ for every $i>0$.
	\end{itemize}
	Two silting objects are equivalent if their additive closures agree.
\end{dfn}

%%%%%%%%%%%%%%%%%%%%%%%%%%%%%%%%%%%%%%%
\subsection{Dg algebras}

For a dg algebra $A$, let $\D(A)$ be the unbounded derived category of right dg $A$-modules, and let $\per(A):=\thick_{\D(A)}(A)$ be its \emph{perfect derived category}. The regular module $A$ is a compact generator of $\D(A)$, and $\per(A)$ is the full subcategory $\D(A)^c$ of compact objects \cite{Ke94}. We use the standard notation
\begin{align*}
	\D^-(A)&:=\{X\in\D(A)\mid H^i(X)=0\text{ for }i\gg0\},\\
	\D^+(A)&:=\{X\in\D(A)\mid H^i(X)=0\text{ for }i\ll0\},\\
	\D^{\mathsf b}(A)&:=\D^-(A)\cap\D^+(A),
\end{align*}
and for $*\in\{-,+,\mathsf b\}$ we put
\[
	\D_{\mathrm{fd}}^*(A):=\{X\in\D^*(A)\mid \dim_k H^i(X)<\infty\text{ for all }i\in\ZZ\}.
\]
The \emph{perfectly valued derived category} of $A$ is $\pvd(A):=\D_{\mathrm{fd}}^{\mathsf b}(A)$; it consists of the dg modules whose underlying complex of $k$-vector spaces is perfect.

%%%%%%%%%%%%%%%%%%%%%%%%%%%%%%%%%%%%%%%%%%%%%%%%%%%%%%%%%%%%%%%%%%%%
\section{Non-positive and positive dg algebras}\label{section:dgalg}

This section prepares the finiteness input for Koszul duality. We first recall the canonical $t$- and weight structures on the non-positive and positive sides and define the distinguished module $S_A$. We then establish the one-sided finiteness criteria used in the duality theorem, together with a perfectness criterion on the positive side. Finally, we introduce pvd-finiteness, the additional finiteness condition on the positive side needed for the duality, and give criteria for it.

\subsection{Canonical structures and distinguished modules}

\begin{dfn}
    A dg algebra $A$ is said to be
    \begin{itemize}
        \item \emph{locally finite} if $\dim_k H^i(A)<\infty$ for every $i\in\ZZ$,
        \item \emph{proper} if $\dim_k H^*(A)<\infty$,
        \item \emph{non-positive} if $H^i(A)=0$ for every $i>0$,
        \item \emph{positive} if $H^i(A)=0$ for every $i<0$ and $H^0(A)$ is semisimple.
    \end{itemize}
\end{dfn}

The terminology ``positive'' follows the usage of \cite{KN13}. The identity
\[
    \Hom_{\D(A)}(A,\Sigma^iA)=H^i(A)
\]
immediately gives the following observations. The dg algebra $A$ is locally finite if and only if $\per(A)$ is Hom-finite, and $A$ is non-positive if and only if the regular module $A$ is a silting object of $\per(A)$. Moreover, $A$ is positive if and only if $\add A$ admits a basic additive generator $P$ whose indecomposable summands form a simple-minded collection of $\per(A)$.

For a non-positive dg algebra $A$ and $n\in\ZZ$, set
\[
\begin{aligned}
	\D(A)^{\le n}&:=\{X\in\D(A)\mid H^i(X)=0\text{ for }i>n\},\\
	\D(A)^{\ge n}&:=\{X\in\D(A)\mid H^i(X)=0\text{ for }i<n\}.
\end{aligned}
\]
Then $(\D(A)^{\le0},\D(A)^{\ge0})$ is the standard $t$-structure on $\D(A)$; see, e.g., \cite{HKM02}. Its truncation functors are denoted by $\tau^{\le n}$ and $\tau^{\ge n}$, and $H^0$ identifies its heart with $\Mod H^0(A)$.

For positive dg algebras we use the canonical weight structure of Keller--Nicol\'as.
\begin{prop}[{\cite[Corollary 5.1]{KN13}}]\label{prop:canonical-weight-positive}
    Let $A$ be a positive dg algebra. Then $\D(A)$ admits a weight structure $(\D(A)_{\ge0},\D(A)_{\le0})$ characterized by
    \begin{align*}
        \D(A)_{\ge0}&=\{X\in\D(A)\mid H^i(X)=0\text{ for every }i<0\},\\
        \D(A)_{\le0}&=\{X\in\D(A)\mid H^i(X)=0\text{ for every }i>0\}.
    \end{align*}
    Moreover, for every $X\in\D(A)$ and $n\in\ZZ$, there exists a weight truncation triangle
    \begin{align*}
        \sigma_{>n}(X)\longrightarrow X\longrightarrow\sigma_{\le n}(X)\longrightarrow
    \end{align*}
    such that the first morphism induces an isomorphism on $H^i$ for $i>n$, while the second induces an isomorphism on $H^i$ for $i\le n$.
\end{prop}

We call this the \emph{canonical weight structure}. We also write $\sigma_{\ge n}:=\sigma_{>n-1}$ and $\sigma_{<n}:=\sigma_{\le n-1}$. Weight truncations are not functorial; we always use triangles as in \cref{prop:canonical-weight-positive}, which are unique up to (non-unique) isomorphism by \cref{rmk:SA-unique}. For a full subcategory $\C\subseteq\D(A)$ we write $\C_{\ge n}:=\C\cap\D(A)_{\ge n}$ and similarly for the other lower-indexed truncation classes.

\begin{dfn}
  Let $A$ be a non-positive or positive dg algebra with $\dim_k H^0(A)<\infty$. We define the \emph{distinguished module} $S_A\in\pvd(A)$ by
  \begin{itemize}
      \item $S_A:=\top H^0(A)$ if $A$ is non-positive, regarded as an object of the heart of the standard $t$-structure;
      \item $S_A:=\sigma_{\le0}(A)$ if $A$ is positive, where $\sigma_{>0}(A)\to A\to\sigma_{\le0}(A)\to$ is a weight truncation triangle as in \cref{prop:canonical-weight-positive}.
  \end{itemize}
\end{dfn}

Thus
\begin{align*}
    H^*(S_A)\simeq
    \begin{cases}
        \top H^0(A),& A\text{ non-positive},\\
        H^0(A),& A\text{ positive},
    \end{cases}
\end{align*}
as graded $H^*(A)$-modules. In the positive case we write $c\colon A\to S_A$ for the morphism in the truncation triangle; it induces an isomorphism on $H^0$.

\begin{rmk}\label{rmk:SA-unique}
Let $A$ be positive. Every $M\in\D(A)$ with $H^{\ne0}(M)=0$ and $H^0(M)\simeq H^0(A)$ as right $H^0(A)$-modules is isomorphic to $\sigma_{\le0}(A)$. Indeed, since $\Hom_{\D(A)}(A,M)=H^0(M)$, there is a morphism $A\to M$ inducing an isomorphism on $H^0$; it factors through $A\to\sigma_{\le0}(A)$ because $\Hom_{\D(A)}(\sigma_{>0}A,M)=0$ by weight orthogonality, and the resulting morphism $\sigma_{\le0}(A)\to M$ is an isomorphism on $H^0=H^*$. In particular, $S_A$ is well defined up to isomorphism, and a similar argument shows that the truncation triangles of \cref{prop:canonical-weight-positive} are unique up to (non-unique) isomorphism.
\end{rmk}

For a positive dg algebra $A$, put
\[
	\H_A:=\Filt(\add A).
\]
If $A$ is moreover locally finite, then $\H_A$ is the heart of a bounded $t$-structure on $\per(A)$, which we call the \emph{canonical $t$-structure} and denote by $(\per(A)^{\le0},\per(A)^{\ge0})$, and $\H_A$ is a length category whose simple objects are, up to isomorphism, the indecomposable summands of $A$ \cite[Theorem~8.1]{KN13}. Explicitly,
\begin{align}
	\per(A)^{\le0}&=\bigcup_{n\ge0} \Sigma^n\H_A\ast\cdots\ast\Sigma\H_A\ast\H_A, \notag\\
	\per(A)^{\ge0}&=\bigcup_{n\ge0} \H_A\ast\Sigma^{-1}\H_A\ast\cdots\ast\Sigma^{-n}\H_A=\per(A)\cap\D(A)_{\ge0}.\label{eq:positive-adjacency}
\end{align}
The last equality holds since $\per(A)^{\ge0}=(\per(A)^{<0})^\perp$ in $\per(A)$, $\per(A)^{<0}$ is the closure of $\{\Sigma^nA\mid n\ge1\}$ under extensions and direct summands, and $\Hom_{\D(A)}(\Sigma^nA,X)=H^{-n}(X)$.

%%%%%%%%%%%%%%%%%%%%%%%%%%%%%%%%%%%%%%
\subsection{One-sided finiteness and perfectness}

\begin{prop}\label{prop:positive-peeling}
    Let $A$ be a positive dg algebra with $\dim_k H^0(A)<\infty$. Then
    \begin{align*}
        \H_A\ast\D(A)_{>0}
        =\{X\in\D(A)_{\ge0}\mid \dim_k\Hom_{\D(A)}(X,S_A)<\infty\}.
    \end{align*}
\end{prop}
\begin{proof}
    The inclusion from left to right follows from weight orthogonality and the fact that $\Hom_{\D(A)}(P,S_A)$ is finite-dimensional for every $P\in\H_A$.

    Conversely, let $X\in\D(A)_{\ge0}$ and suppose that $d:=\dim_k\Hom_{\D(A)}(X,S_A)$ is finite. We argue by induction on $d$. Since $H^0(A)$ is semisimple, there exist $P\in\Add A$ and a morphism $f\colon P\to X$ such that $H^0(f)\colon H^0(P)\to H^0(X)$ is an isomorphism. Complete $f$ to a triangle
    \begin{align*}
        P\xto{f}X\longrightarrow Y\longrightarrow.
    \end{align*}
    Since $P,X\in\D(A)_{\ge0}$ and $H^0(f)$ is an isomorphism, the long exact cohomology sequence gives $Y\in\D(A)_{\ge0}$. Since $\Sigma^{-1}Y\in\D(A)_{>0}$ and $S_A\in\D(A)_{\le0}$, weight orthogonality gives $\Hom_{\D(A)}(\Sigma^{-1}Y,S_A)=0$; moreover, $\Hom_{\D(A)}(\Sigma P,S_A)=0$. Hence applying $\Hom_{\D(A)}(-,S_A)$ gives a short exact sequence
    \begin{align*}
        0\longrightarrow\Hom_{\D(A)}(Y,S_A)
        \longrightarrow\Hom_{\D(A)}(X,S_A)
        \longrightarrow\Hom_{\D(A)}(P,S_A)\longrightarrow0.
    \end{align*}
    Thus $\Hom_{\D(A)}(P,S_A)$ is finite-dimensional, which forces $P\in\add A$. If $P=0$, then $H^0(X)=0$, hence $X\in\D(A)_{>0}$. If $P\neq0$, then $\Hom_{\D(A)}(P,S_A)\neq0$, so
    \begin{align*}
        \dim_k\Hom_{\D(A)}(Y,S_A)<d.
    \end{align*}
    By induction, $Y\in\H_A\ast\D(A)_{>0}$. Since $P\in\add A\subseteq\H_A$ and $\H_A$ is extension-closed, we conclude $X\in\H_A\ast\D(A)_{>0}$.
\end{proof}

\begin{prop}\label{prop:one-sided-finiteness}
Let $A$ be a dg algebra with $\dim_kH^0(A)<\infty$ and suppose that $\pvd(A)$ is Hom-finite.
\begin{itemize}
    \item[(1)] If $A$ is non-positive, then $A$ is locally finite and
    \[
        \Dmfd(A)=\{X\in\D^-(A)\mid \RHom_A(X,S_A)\in\Dpfd(k)\}.
    \]
    \item[(2)] If $A$ is positive, then $A$ is locally finite and
    \[
        \Dpfd(A)=\{X\in\D^+(A)\mid \RHom_A(X,S_A)\in\Dmfd(k)\}.
    \]
\end{itemize}
\end{prop}
\begin{proof}
In either case put
\[
    \Q:=\{X\mid \dim_k\Hom_{\D(A)}(X,\Sigma^iS_A)<\infty\text{ for all }i\in\ZZ\},
\]
inside $\D^-(A)$ in (1) and $\D^+(A)$ in (2). For such $X$, $t$-orthogonality in (1) (resp.\ weight orthogonality in (2)) gives $\Hom_{\D(A)}(X,\Sigma^iS_A)=0$ for $i\ll0$ (resp.\ $i\gg0$), so $\Q$ is the right-hand side of the asserted equality. We have $S_A\in\Q$ by Hom-finiteness, and $A\in\Q$ by $\dim_k H^0(A)<\infty$. It therefore suffices to show that $\Q=\Dmfd(A)$ in (1) and $\Q=\Dpfd(A)$ in (2); in particular, $A$ is then locally finite.

For (1), let $0\neq X\in\mathcal Q$, and let $n$ be the largest integer such that $H^n(X)\neq0$. The standard truncation triangle and $t$-orthogonality give
\[
	\Hom_{\D(A)}(X,\Sigma^{-n}S_A)\simeq\Hom_{H^0(A)}(H^n(X),S_A).
\]
Since the left-hand side is finite-dimensional and $S_A=\top H^0(A)$, the top of $H^n(X)$ is finite-dimensional. As the Jacobson radical of $H^0(A)$ is nilpotent, it follows that $H^n(X)$ is finite-dimensional. Thus $\Sigma^{-n}H^n(X)\in\pvd(A)\subseteq\mathcal Q$, and the truncation triangle shows that $\tau^{<n}X\in\mathcal Q$. Repeating this argument in lower degrees, we conclude that $X\in\Dmfd(A)$. Conversely, for $X\in\Dmfd(A)$ and fixed $i\in\ZZ$, take $m\ll0$. Then $\tau^{\ge m}X\in\pvd(A)$ and truncation orthogonality gives
\[
    \Hom(X,\Sigma^iS_A)\simeq\Hom(\tau^{\ge m}X,\Sigma^iS_A),
\]
so $X\in\mathcal Q$. Thus $\mathcal Q=\Dmfd(A)$.

For (2), let $X\in\mathcal Q$. We show by induction on $j$ that $H^j(X)$ is finite-dimensional; this is clear for $j\ll0$ since $X\in\D^+(A)$. Suppose that $H^i(X)$ is finite-dimensional for every $i<j$. Let $\sigma_{\ge j}X\to X\to\sigma_{<j}X\to$ be a weight truncation. Since $X\in\D^+(A)$, the induction hypothesis gives $\sigma_{<j}X\in\pvd(A)\subseteq\mathcal Q$, and hence $\sigma_{\ge j}X\in\mathcal Q$. Put $Z:=\Sigma^j\sigma_{\ge j}X$. Then we have $Z\in\Q\cap\D(A)_{\ge0}$ and $\dim_k\Hom_{\D(A)}(Z,S_A)<\infty$. By \cref{prop:positive-peeling}, we have $Z\in\H_A\ast\D(A)_{>0}$. It follows that $H^j(X)\simeq H^0(Z)$ is finite-dimensional, and therefore $\Q\subseteq\Dpfd(A)$. Conversely, if $X\in\Dpfd(A)$ and $i$ is fixed, choose $m\gg0$ and a weight truncation triangle
\[
    \sigma_{>m}X\longrightarrow X\longrightarrow\sigma_{\le m}X\longrightarrow.
\]
Then $\sigma_{\le m}X\in\pvd(A)$, and weight orthogonality gives
\[
    \Hom(X,\Sigma^iS_A)\simeq\Hom(\sigma_{\le m}X,\Sigma^iS_A).
\]
Hence $\mathcal Q=\Dpfd(A)$.
\end{proof}

\begin{prop}\label{prop:perfect-positive}
    Let $A$ be a positive dg algebra with $\dim_k H^0(A)<\infty$. Then
    \begin{align*}
        \per(A)=\{X\in\D^+(A)\mid\RHom_A(X,S_A)\in\per(k)\}.
    \end{align*}
\end{prop}
\begin{proof}
    Only the reverse inclusion needs proof. Let $X\in\D^+(A)$ with $\RHom_A(X,S_A)\in\per(k)$, and set $\delta(X)=\dim_k H^*\RHom_A(X,S_A)$. If $X\neq0$ and $r$ is minimal with $H^r(X)\neq0$, \cref{prop:positive-peeling} gives a triangle
    \[
        \Sigma^{-r}P\longrightarrow X\longrightarrow Y\longrightarrow
    \]
    with $0\neq P\in\H_A$ and $H^i(Y)=0$ for every $i\le r$. Since $\RHom_A(P,S_A)$ is concentrated in degree $0$ for $P\in\H_A$, the complex $\RHom_A(\Sigma^{-r}P,S_A)$ is concentrated in degree $-r$. Since moreover $H^i\RHom_A(Y,S_A)=0$ for $i\ge-r$, the long exact cohomology sequence gives
    \[
        \delta(X)=\delta(\Sigma^{-r}P)+\delta(Y),
    \]
    with $\delta(\Sigma^{-r}P)=\dim_k\Hom_{\D(A)}(P,S_A)>0$, since $\dim_k\Hom_{\D(A)}(-,S_A)$ is additive on extensions in $\H_A$ and positive on nonzero objects of $\add A$. Thus we have
    \[
      0\le\delta(Y)<\delta(X).
    \]
    Induction on $\delta(X)$ gives $Y\in\per(A)$, and therefore $X\in\per(A)$. 
\end{proof}

\begin{rmk}
For a non-positive dg algebra $A$ with finite-dimensional $H^0(A)$, the analogous characterization
\[
    \per(A)=\{X\in\D^-(A)\mid \RHom_A(X,S_A)\in\per(k)\}
\]
is due to Jin \cite[Lemma~2.10]{Ji23}, where it is stated with $S_A$ replaced by all objects of $\pvd(A)=\thick(S_A)$.
\end{rmk}

%%%%%%%%%%%%%%%%%%%%%%%%%%%%%%%%%%%%%%%%%%%%%%%%%%%%%%%%%%%%%%%%%%%%%%%%%%%%%%%%%%%%%%%%%%%%%%%%%%%%%%%%%%%%%%%%%%%%%%%%%%%%%%%%%%%%%%%%%%%%%%%%%%%%%%%%%%%%%%%%%%%%%%%%%%%%%%%%%%

\subsection{pvd-finite positive dg algebras}\label{subsection:pvd-finite}

\begin{prop}\label{prop:pvd-basic}
Let $A$ be a locally finite positive dg algebra. The following conditions are equivalent:
\begin{itemize}
    \item[(1)] $\dim_k\End_{\D(A)}(S_A)<\infty$;
    \item[(2)] $S_A\in\H_A\ast\Dpfd(A)_{>0}$;
    \item[(3)] $\H_A$ has an injective cogenerator.
\end{itemize}
\end{prop}
\begin{proof}
By \cref{prop:positive-peeling}, (1) holds if and only if $S_A\in\H_A\ast\D(A)_{>0}$. In any triangle $P\to S_A\to K\to$ with $P\in\H_A$ and $K\in\D(A)_{>0}$, the object $K$ lies in $\D^+(A)$ and has finite-dimensional cohomology in each degree, since $A$ is locally finite. Hence $(1)\Leftrightarrow(2)$.

Assume (2), and choose a triangle
\[
	I\longrightarrow S_A\longrightarrow K\longrightarrow,
	\qquad I\in\H_A,\quad K\in\Dpfd(A)_{>0}.
\]
Since $H^0(K)=0=H^1(S_A)$, applying $\Hom_{\D(A)}(A,-)$ gives $\Ext^1_{\H_A}(A,I)=0$. As $\H_A$ is a length category whose simple objects are the indecomposable summands of $A$, the object $I$ is injective. Moreover, $H^{-1}(K)=0=H^0(K)$ gives $H^0(I)\simeq H^0(S_A)\simeq H^0(A)$, so $\Hom_{\D(A)}(P,I)\neq0$ for every simple object $P$ of $\H_A$. Thus every simple object embeds into $I$; since every object of the length category $\H_A$ is an essential extension of its socle, $I$ is an injective cogenerator of $\H_A$.

Conversely, suppose that $\H_A$ has an injective cogenerator. Since $\H_A$ is Hom-finite, the usual Morita theory identifies $\H_A^o$ with $\mod\Gamma$ for a finite-dimensional algebra $\Gamma$; in particular, injective envelopes exist in $\H_A$. Let $u\colon A\to I$ be the injective envelope of $A$. Since $\Hom(\Sigma^{-1}\H_A,S_A)=0$, the morphism $c\colon A\to S_A$ extends to $f\colon I\to S_A$ with $fu=c$:
\[
\begin{tikzcd}[column sep=4.2em,row sep=2.8em]
A \arrow[r,hook,"u"] \arrow[dr,"c"'] & I \arrow[d,dashed,"f"] \\
& S_A.
\end{tikzcd}
\]
Since $A$ is semisimple in $\H_A$, the injective envelope $u$ identifies $A$ with the socle of $I$, so composition with $u$ gives an isomorphism
\[
    \Hom(A,A)\xrightarrow{\sim}\Hom(A,I).
\]
The map $\Hom(A,A)\to\Hom(A,S_A)$ induced by $c$ is also an isomorphism. Since $fu=c$, it follows that $\Hom(A,f)$ is an isomorphism. Put $K=\cone(f)$. The long exact sequence of the triangle $I\to S_A\to K\to$, together with $\Hom_{\D(A)}(A,\Sigma I)=\Ext^1_{\H_A}(A,I)=0$, gives $H^{-1}(K)=0=H^0(K)$. Since
\[
    K\in\add S_A\ast\Sigma\H_A\subseteq\Dpfd(A)_{\ge-1},
\]
we have $K\in\Dpfd(A)_{>0}$, proving (2).
\end{proof}

\begin{dfn}\label{dfn:pvd-finite}
A locally finite positive dg algebra $A$ is \emph{pvd-finite} if it satisfies the equivalent conditions of \cref{prop:pvd-basic}. The terminology anticipates \cref{thm:main}(3), which identifies this condition with Hom-finiteness of $\pvd(A)$.
\end{dfn}

The condition is not automatic, even for very small positive dg algebras.
\begin{eg}\label{exa:non-pvd-finite}
Let $A=k[y]/(y^2)$ with $|y|=1$ and zero differential. Then $A$ is a proper positive dg algebra and $S_A=k$. A standard bar calculation gives
\[
    A^!=\REnd_A(k)\simeq k[\![x]\!],
    \qquad |x|=0,
\]
so $A^!$ is non-positive but not locally finite, and $\End_{\D(A)}(S_A)=H^0(A^!)$ is infinite-dimensional. Hence $A$ is not pvd-finite.
\end{eg}

The next criterion shows that pvd-finiteness follows from a directed degree-one extension pattern among the summands of the regular module.

\begin{prop}\label{prop:directed-pvd}
	Let $A$ be a locally finite positive dg algebra, and suppose that the regular dg module admits a decomposition $A\simeq\bigoplus_{i=1}^nP_i$ in $\add A$. If
	\[
		\Hom_{\D(A)}(P_i,\Sigma P_j)=0\qquad(1\le j\le i\le n),
	\]
	then $A$ is pvd-finite.
\end{prop}

The proposition follows from the following general fact about length categories.

\begin{lem}\label{lem:directed-length}
	Let $\H$ be a $k$-linear abelian length category with finitely many simple objects $S_1,\ldots,S_r$ up to isomorphism. Suppose that $\Ext^1_{\H}(S_i,S_j)$ is finite-dimensional for all $i,j$ and that
	\[
		\Ext^1_{\H}(S_i,S_j)=0\qquad(1\le j\le i\le r).
	\]
	Then $\H$ admits an injective cogenerator.
\end{lem}
\begin{proof}
	It suffices to show that each $S_j$ is contained in some injective object. Fix $j$ and set $I_j^{(j)}=S_j$. For $s=j-1,\ldots,1$, choose an extension
	\[
		0\longrightarrow I_j^{(s+1)}\longrightarrow I_j^{(s)}\longrightarrow Q_s\longrightarrow0,
		\qquad Q_s\in\add S_s,
	\]
	such that the connecting map
	\[
		\Hom_{\H}(S_s,Q_s)\longrightarrow\Ext^1_{\H}(S_s,I_j^{(s+1)})
	\]
	is surjective. Such an extension is obtained by choosing generators of the finite-dimensional Ext-space and taking the corresponding universal extension.

	The long exact Ext-sequence and the assumed vanishing give, by descending induction,
	\[
		\Ext^1_{\H}(S_i,I_j^{(s)})=0\qquad(i\ge s).
	\]
	Thus $I_j:=I_j^{(1)}$ is injective and contains $S_j$.
\end{proof}

\begin{proof}[Proof of \cref{prop:directed-pvd}]
The simple objects of $\H_A$ are the indecomposable summands of the $P_i$, and $\Ext^1_{\H_A}(X,Y)\simeq\Hom_{\D(A)}(X,\Sigma Y)$ is finite-dimensional for $X,Y\in\add A$. If we enumerate the simple objects compatibly with the order of the $P_i$, \cref{lem:directed-length} shows that $\H_A$ has an injective cogenerator, so $A$ is pvd-finite by \cref{prop:pvd-basic}.
\end{proof}

\begin{cor}\label{cor:H1-zero}
If $A$ is a locally finite positive dg algebra with $H^1(A)=0$, then $A$ is pvd-finite.
\end{cor}
\begin{proof}
Apply \cref{prop:directed-pvd} with $n=1$ and $P_1=A$, since $\Hom_{\D(A)}(A,\Sigma A)=H^1(A)=0$.
\end{proof}

%%%%%%%%%%%%%%%%%%%%%%%%%%%%%%%%%%%%%%%%%%%%%%%%%%%%%%%%%%%%%%%%%%%%%%%%%%%%%%%%%%%%%%%%%%%%%%%%%%%%%%%%%%%%%%%%%%%%%%%
\section{Koszul duality}\label{section:koszul-duality}

In this section we establish Koszul duality between locally finite non-positive and pvd-finite positive dg algebras. We then characterize pvd-finiteness and derive the Morita characterization, the description of length hearts as module categories and the ST correspondence. Finally, we prove double duality, which yields the bijection of \cref{intro:KD-on-alg}, and establish an adjunction for the Koszul dual functor.

Let $\Cdg(A)$ be the dg category of dg $A$-modules and let $\Ddg(A)$ be its Drinfeld dg quotient by the full dg subcategory of acyclic modules; see \cite{Dri04,Ke06}. Thus $H^0(\Ddg(A))\simeq\D(A)$. For a non-positive or positive dg algebra $A$ with $\dim_kH^0(A)<\infty$, put
\[
    A^!:=\End_{\Ddg(A)}(S_A)\simeq\REnd_A(S_A),
    \qquad
    B_A:=A^{!,o}.
\]

Composing the dg functor
\[
    \Hom_{\Ddg(A)}(-,S_A)
    \colon \Ddg(A)\longrightarrow\Cdg(B_A)^o
\]
with the opposite of the quotient dg functor $Q_{B_A}\colon\Cdg(B_A)\to\Ddg(B_A)$, we obtain the \emph{Koszul dual functor}
\[
    \Phi_A:=Q_{B_A}^o\circ
    \Hom_{\Ddg(A)}(-,S_A)
    \colon\Ddg(A)\longrightarrow\Ddg(B_A)^o.
\]
We also write $\Phi_A$ for the induced triangle functor
\[
    H^0(\Phi_A)=\RHom_A(-,S_A)
    \colon\D(A)\longrightarrow\D(B_A)^o.
\]

When there is no danger of confusion, we write simply $B=B_A$ and $\Phi=\Phi_A$.

\begin{rmk}\label{rmk:opposite-convention}
Since $A^!$ acts on $\RHom_A(X,S_A)$ from the left, $\Phi_A$ takes values in $\D(A^{!,o})^o$. The $k$-dual gives an equivalence $\DD\colon\pvd(A)\xrightarrow{\sim}\pvd(A^o)^o$ with $\DD(S_A)\simeq S_{A^o}$ (in the positive case by \cref{rmk:SA-unique}); hence $A^{o,!}$ and $A^{!,o}$ are quasi-isomorphic.
\end{rmk}

\subsection{Finite-object duality}

\begin{prop}\label{prop:finite-koszul-duality}
Let $A$ be a non-positive or positive dg algebra with finite-dimensional zeroth cohomology, and put $B=A^{!,o}$.
\begin{itemize}
    \item[(1)] If $A$ is non-positive, then $A^!$ is positive, and a set of representatives of the isomorphism classes of indecomposable summands of $S_A$ forms a simple-minded collection of $\pvd(A)$.
    \item[(2)] If $A$ is positive, then $A^!$ is non-positive and $S_A$ is a silting object of $\pvd(A)$.
    \item[(3)] For $X\in\D(A)$ and $Y\in\pvd(A)$, the canonical morphism
    \[
        \RHom_A(X,Y)\longrightarrow\RHom_B(\Phi Y,\Phi X)
    \]
    is a quasi-isomorphism. In particular, $\Phi$ induces a triangle equivalence
    \[
        \Phi\colon\pvd(A)\xrightarrow{\sim}\per(B)^o,
        \qquad S_A\longmapsto B.
    \]
\end{itemize}
\end{prop}
\begin{proof}
If $A$ is non-positive, the standard $t$-structure restricts to $\pvd(A)$ with heart $\mod H^0(A)$. Its simple objects are, up to isomorphism, the indecomposable summands of $S_A$. Finite Postnikov towers give $\thick(S_A)=\pvd(A)$, while $t$-orthogonality gives the required negative vanishing. Thus these summands form a simple-minded collection and $A^!=\REnd_A(S_A)$ is positive. Part~(2) follows from \cite[Proposition~5.6 and Lemma~6.2]{KN13}; in particular $\thick(S_A)=\pvd(A)$ and $H^i(A^!)=0$ for every $i>0$.

For (3), fix $X$. When $Y=S_A$, the morphism is the tautological identification obtained from $\Phi(S_A)=B$. The full subcategory of objects $Y$ for which it is a quasi-isomorphism is thick, hence contains $\thick(S_A)=\pvd(A)$. Taking $X,Y\in\pvd(A)$ gives full faithfulness, and the essential image is $\thick(B)=\per(B)$.
\end{proof}

\subsection{The duality theorem}\label{subsection:KD-theorem}

The passage from finite objects to the one-sided categories is provided by the following completion statement. In the positive case we use compatible weight truncations as in \cite[Lemma~6.3]{KN13}.

\begin{lem}\label{lem:KD-completion}
\textup{(1)} If $A$ is a locally finite non-positive dg algebra and $Y\in\Dmfd(A)$, then
    \[
        Y\simeq\holim_m\tau^{\ge-m}Y,
        \qquad
        \hocolim_m\Phi(\tau^{\ge-m}Y)\simeq\Phi(Y).
    \]

\textup{(2)} If $A$ is a locally finite positive dg algebra and $Y\in\Dpfd(A)$, then
    \[
        Y\simeq\holim_m\sigma_{\le m}Y,
        \qquad
        \hocolim_m\Phi(\sigma_{\le m}Y)\simeq\Phi(Y).
    \]
\end{lem}
\begin{proof}
Write $Y_m=\tau^{\ge-m}Y$ in (1) and $Y_m=\sigma_{\le m}Y$ in (2), with transition maps compatible with the morphisms $Y\to Y_m$. For each fixed $i$, the inverse system $H^i(Y_m)$ is eventually constant with value $H^i(Y)$. The Milnor exact sequence therefore has vanishing $\varprojlim^1$ term and gives
\[
    Y\xrightarrow{\sim}\holim_mY_m.
\]

Contravariance turns the inverse system $(Y_m)$ into a direct system $(\Phi Y_m)$, and the morphisms $Y\to Y_m$ induce a morphism
\[
    \hocolim_m\Phi(Y_m)\longrightarrow\Phi(Y).
\]
Fix $i$. For $m\gg0$, the truncation triangle gives
\[
    \Hom(Y_m,\Sigma^iS_A)\xrightarrow{\sim}\Hom(Y,\Sigma^iS_A);
\]
since cohomology commutes with sequential homotopy colimits and
\[
    H^i(\Phi X)=\Hom(X,\Sigma^iS_A),
\]
the morphism $\hocolim_m\Phi(Y_m)\to\Phi(Y)$ is a quasi-isomorphism.
\end{proof}

\begin{thm}[Koszul duality]\label{thm:KD}
\textup{(1)} Let $B$ be a locally finite non-positive dg algebra and put $A=B^{!,o}$. Then $A$ is a pvd-finite positive dg algebra, $\Phi_B(B)\simeq S_A$, and $\Phi_B$ induces triangle equivalences
\[
\begin{tikzcd}[column sep=4.6em,row sep=0.55em]
\per(B) \arrow[d,phantom,"\rotatebox{-90}{$\subseteq$}"{description}] \arrow[r,"\sim"]
    & \pvd(A)^o \arrow[d,phantom,"\rotatebox{-90}{$\subseteq$}"{description}] \\
\Dmfd(B) \arrow[r,"\sim"]
    & \Dpfd(A)^o \\
\pvd(B) \arrow[u,phantom,"\rotatebox{90}{$\subseteq$}"{description}] \arrow[r,"\sim"]
    & \per(A)^o \arrow[u,phantom,"\rotatebox{90}{$\subseteq$}"{description}].
\end{tikzcd}
\]
\textup{(2)} Let $A$ be a pvd-finite positive dg algebra and put $B=A^{!,o}$. Then $B$ is a locally finite non-positive dg algebra, $\Phi_A(A)\simeq S_B$, and $\Phi_A$ induces triangle equivalences
\[
\begin{tikzcd}[column sep=4.6em,row sep=0.55em]
\per(A) \arrow[d,phantom,"\rotatebox{-90}{$\subseteq$}"{description}] \arrow[r,"\sim"]
    & \pvd(B)^o \arrow[d,phantom,"\rotatebox{-90}{$\subseteq$}"{description}] \\
\Dpfd(A) \arrow[r,"\sim"]
    & \Dmfd(B)^o \\
\pvd(A) \arrow[u,phantom,"\rotatebox{90}{$\subseteq$}"{description}] \arrow[r,"\sim"]
    & \per(B)^o \arrow[u,phantom,"\rotatebox{90}{$\subseteq$}"{description}].
\end{tikzcd}
\]
Moreover, $\Phi_B(S_B)\simeq A$ in \textup{(1)} and $\Phi_A(S_A)\simeq B$ in \textup{(2)}.
\end{thm}
\begin{proof}
For uniform notation, put
\[
    (C,D,\Phi)=
    \begin{cases}
        (B,A,\Phi_B),&\text{in \textup{(1)}},\\
        (A,B,\Phi_A),&\text{in \textup{(2)}}.
    \end{cases}
\]
Thus $C$ is the source algebra and $D=C^{!,o}$ the opposite of its Koszul dual. By \cref{prop:finite-koszul-duality}, in both cases
\begin{equation}\label{eq:finite-duality}
    \Phi\colon\pvd(C)\xrightarrow{\sim}\per(D)^o.
\end{equation}

We first extend \eqref{eq:finite-duality} to the relevant one-sided category. Let
\[
    \mathcal C=
    \begin{cases}
        \Dmfd(B),&\text{in \textup{(1)}},\\
        \Dpfd(A),&\text{in \textup{(2)}},
    \end{cases}
\]
and, for $Y\in\mathcal C$, put $Y_m=\tau^{\ge-m}Y$ in \textup{(1)} and $Y_m=\sigma_{\le m}Y$ in \textup{(2)}. Then $Y_m\in\pvd(C)$ and \cref{lem:KD-completion} gives
\[
    Y\simeq\holim_mY_m,
    \qquad
    \Phi Y\simeq\hocolim_m\Phi Y_m.
\]

For every $X\in\D(C)$, the dg functor $\Phi$ induces a commutative square in $\D(k)$:

\[
\begin{tikzcd}[column sep=6.4em,row sep=3.0em]
\RHom_C(X,Y) \arrow[r] \arrow[d,"\sim"'sloped]
    & \RHom_D(\Phi Y,\Phi X) \arrow[d,"\sim"sloped] \\
\holim_m\RHom_C(X,Y_m) \arrow[r,"\sim"']
    & \holim_m\RHom_D(\Phi Y_m,\Phi X).
\end{tikzcd}
\]
The horizontal maps are induced by $\Phi$. The bottom map is a quasi-isomorphism by \cref{prop:finite-koszul-duality}. Hence the top map is a quasi-isomorphism, and $\Phi$ is fully faithful on $\mathcal C$.

We next identify the image of the regular module. In \textup{(1)}, the canonical morphism $B\to S_B$ induces a morphism $A\simeq\Phi(S_B)\to\Phi(B)$, which on $H^0$ is the isomorphism $\Hom_{\D(B)}(S_B,S_B)\to\Hom_{\D(B)}(B,S_B)$. Since $\Phi(B)=\RHom_B(B,S_B)$ has cohomology concentrated in degree zero, \cref{rmk:SA-unique} gives $\Phi(B)\simeq S_A$.

In \textup{(2)}, \cref{prop:pvd-basic} gives a triangle
\[
    I\longrightarrow S_A\longrightarrow K\longrightarrow,
\]
with $I\in\H_A$ and $K\in\Dpfd(A)_{>0}$. Applying $\Phi_A$, we have $\Phi_A(K)\in\D(B)^{\le-1}$ by weight orthogonality, while $\Phi_A(I)$ has cohomology only in degree $0$; hence $B\to\Phi_A(I)$ induces $H^0(B)\simeq\Phi_A(I)$. The Koszul dual functor induces a fully faithful functor $\per(A)\to\pvd(B)^o$, which maps $\H_A$ into $\mod H^0(B)$. Since $\Phi$ is exact on $\H_A$ and fully faithful on $\per(A)$, and $\Ext^1$ in both hearts is given by $\Hom(-,\Sigma-)$, the image $\Phi(\H_A)$ is a wide subcategory of $\mod H^0(B)$ and contains the regular module $H^0(B)\simeq\Phi(I)$. It follows that $\Phi$ restricts to an anti-equivalence $\H_A\to\mod H^0(B)$; in particular, $\Phi(A)$ is semisimple. For every simple object $P$ of $\H_A$, since $c\colon A\to S_A$ induces an isomorphism on $H^0$, we have
\begin{align*}
	\Hom_{H^0(B)}(\Phi A,\Phi P)&\simeq\Hom_{\D(A)}(P,A)\xrightarrow[\sim]{\ c_*\ }\Hom_{\D(A)}(P,S_A)\\
	&=H^0(\Phi P)\simeq\Hom_{H^0(B)}(H^0(B),\Phi P).
\end{align*}
As the objects $\Phi P$ exhaust the simple $H^0(B)$-modules, comparing multiplicities gives $\Phi(A)\simeq\top H^0(B)$, that is, $\Phi_A(A)\simeq S_B$.

Thus $\Phi(C)\simeq S_D$ in either case, and since $\pvd(D)=\thick(S_D)$ by \cref{prop:finite-koszul-duality}, we obtain an equivalence
\begin{equation}\label{eq:perfect-pvd-duality}
    \Phi\colon\per(C)\xrightarrow{\sim}\pvd(D)^o.
\end{equation}
We now check the local finiteness of $D$. In either case, $\per(C)$ is Hom-finite by local finiteness of $C$, and so is $\pvd(D)$. Since $H^0(D)$ is finite-dimensional, \cref{prop:one-sided-finiteness} shows that $D$ is locally finite. Hence $\per(D)$ is Hom-finite, and so is $\pvd(C)\simeq\per(D)^o$ by \eqref{eq:finite-duality}. In \textup{(1)}, $A$ is moreover pvd-finite: by full faithfulness, $\End_{\D(A)}(S_A)\simeq\End_{\D(A)}(\Phi_BB)\simeq H^0(B)^o$ is finite-dimensional, so \cref{prop:pvd-basic} applies.

Using
\[
    H^i(\Phi X)=\Hom_{\D(C)}(X,\Sigma^iS_C),
\]
\cref{prop:one-sided-finiteness} now shows that $\Phi_B$ takes $\Dmfd(B)$ to $\Dpfd(A)^o$ in \textup{(1)}, while $\Phi_A$ takes $\Dpfd(A)$ to $\Dmfd(B)^o$ in \textup{(2)}. It remains only to prove density.

Let $Z$ belong to $\Dpfd(D)$ in (1) and $\Dmfd(D)$ in (2), and take $Z_m=\sigma_{\le m}Z$ in \textup{(1)} and $Z_m=\tau^{\ge-m}Z$ in \textup{(2)}. Then $Z_m\in\pvd(D)$ and $Z\simeq\holim_mZ_m$. Lift $(Z_m)$ through a quasi-inverse of \eqref{eq:perfect-pvd-duality} to a direct system $(X_m)$ in $\per(C)$. We choose compatible identifications $\Phi(X_m)\simeq Z_m$, so that the two systems have the form
\[
\begin{tikzcd}[column sep=3.8em,row sep=2.5em]
X_0 \arrow[r] \arrow[d,mapsto,"\Phi"']
    & X_1 \arrow[r] \arrow[d,mapsto,"\Phi"']
    & X_2 \arrow[r] \arrow[d,mapsto,"\Phi"']
    & \cdots \\
Z_0 & Z_1 \arrow[l] & Z_2 \arrow[l] & \cdots \arrow[l].
\end{tikzcd}
\]
Thus contravariance turns the direct system upstairs into the inverse truncation tower downstairs. Put $X=\hocolim_mX_m$. Then
\[
    \Phi(X)\simeq\holim_m\Phi(X_m)\simeq Z.
\]
We claim that $X\in\D^-(B)$ in \textup{(1)} and $X\in\D^+(A)$ in \textup{(2)}. In \textup{(1)}, choose $r$ with $H^{<r}(Z)=0$. Then $Z_m\in\D(A)_{\ge r}$ for all $m$, and full faithfulness of $\Phi$ on $\per(B)$ together with $\Phi(B)\simeq S_A$ gives
\[
	H^i(X_m)\simeq\Hom_{\D(B)}(B,\Sigma^iX_m)\simeq\Hom_{\D(A)}(Z_m,\Sigma^iS_A),
\]
which vanishes for $i>-r$ by weight orthogonality. In \textup{(2)}, choose $r$ with $H^{>r}(Z)=0$. Then $Z_m\in\D(B)^{\le r}$, and similarly $H^i(X_m)\simeq\Hom_{\D(B)}(Z_m,\Sigma^iS_B)=0$ for $i<-r$ by $t$-orthogonality. Since cohomology commutes with sequential homotopy colimits, the claim follows. As $\RHom_C(X,S_C)=\Phi(X)\simeq Z$, \cref{prop:one-sided-finiteness} gives $X\in\Dmfd(B)$ in \textup{(1)} and $X\in\Dpfd(A)$ in \textup{(2)}. This proves density and completes the proof.
\end{proof}

%%%%%%%%%%%%%%%%%%%%%%%%%%%%%%%%%%%%%%%%%%%%%%%%%%%%%%%%%%%%%%%%
\subsection{Characterizations of pvd-finiteness}

Several of the finiteness equivalences in the following theorem were suggested by Dong Yang. 

\begin{thm}[Characterizations of pvd-finiteness]\label{thm:main}
Let $A$ be a locally finite positive dg algebra. The following conditions are equivalent:
\begin{itemize}
    \item[(1)] $\H_A$ has an injective cogenerator;
    \item[(2)] $A^!$ is locally finite;
    \item[(3)] $\pvd(A)$ is Hom-finite;
    \item[(4)] $\dim_k\Hom_{\D(A)}(X,Y)<\infty$ for every $X\in\Dpfd(A)$ and $Y\in\pvd(A)$;
    \item[(5)] $\End_{\D(A)}(S_A)$ is finite-dimensional;
    \item[(6)] $(\per(A)^{\le0},\Dpfd(A)_{\ge0})$ is a $t$-structure on $\Dpfd(A)$ with heart $\H_A$.
\end{itemize}
\end{thm}
\begin{proof}
By \cref{prop:pvd-basic}, (1) and (5) are equivalent; both say that $A$ is pvd-finite. By \cref{prop:finite-koszul-duality}, (2) and (3) are equivalent.

(5)$\To$(2): This is \cref{thm:KD}(2).

(3)$\To$(4): Since $\pvd(A)$ is Hom-finite, \cref{prop:one-sided-finiteness} implies that $\RHom_A(X,Y)\in\Dmfd(k)$ for every $X\in\Dpfd(A)$ and $Y\in\pvd(A)=\thick(S_A)$.

(4)$\To$(6): By (4), $\Hom_{\D(A)}(X,S_A)$ is finite-dimensional for every $X\in\Dpfd(A)_{\ge0}$, so \cref{prop:positive-peeling} gives $\Dpfd(A)_{\ge0}=\H_A\ast\Dpfd(A)_{>0}$. Iterating this decomposition yields
\[
	\Dpfd(A)=\per(A)^{\le0}\ast\Dpfd(A)_{>0}.
\]
Moreover, $\per(A)^{\le0}$ is closed under $\Sigma$, and since $\Hom_{\D(A)}(\Sigma^nA,Y)=H^{-n}(Y)$,
\[
	\Hom_{\D(A)}(\per(A)^{\le0},\Dpfd(A)_{>0})=0.
\]
Thus $(\per(A)^{\le0},\Dpfd(A)_{\ge0})$ is a $t$-structure on $\Dpfd(A)$. By \eqref{eq:positive-adjacency}, its heart is
\[
	\per(A)^{\le0}\cap\Dpfd(A)_{\ge0}=\per(A)^{\le0}\cap\bigl(\per(A)\cap\D(A)_{\ge0}\bigr)=\per(A)^{\le0}\cap\per(A)^{\ge0}=\H_A.
\]

(6)$\To$(5): There is a truncation triangle
\[
	I\to S_A\to K\to
\]
with $I\in\per(A)^{\le0}$ and $K\in\Dpfd(A)_{>0}$. Since $\Hom_{\D(A)}(K,S_A)=0=\Hom_{\D(A)}(\Sigma^{-1}K,S_A)$ by weight orthogonality, applying $\Hom_{\D(A)}(-,S_A)$ gives
\[
	\dim \End_{\D(A)}(S_A)=\dim \Hom_{\D(A)}(I,S_A)<\infty.
\]
\end{proof}

The pvd-finiteness of a positive dg algebra can be read off from a minimal tensor model.

\begin{prop}\label{prop:augmented-minimal-model}
	Let $S$ be a finite-dimensional semisimple $k$-algebra, and let
	\[
		M=(T_S(V),d),\qquad V=\bigoplus_{q\ge1}V^q,
	\]
	be a locally finite dg $S$-algebra, where $V$ is a graded $S$-bimodule and $d(S)=0$. Assume that $M$ is \emph{relatively semifree}, that is, $V$ admits an exhaustive filtration $0=V(-1)\subseteq V(0)\subseteq V(1)\subseteq\cdots$ by graded sub-bimodules with $d(V(n))\subseteq T_S(V(n-1))$ for all $n\ge0$, and \emph{minimal}, that is,
	\[
		d(V)\subseteq\bigoplus_{m\ge2}V^{\otimes_S m}.
	\]
	We then call $M$ a \emph{minimal tensor model}. Then the following conditions are equivalent:
	\begin{enumerate}[label=(\arabic*),leftmargin=2.4em]
		\item $M$ is pvd-finite;
		\item $\dim_k V^1<\infty$;
		\item $\dim_k V^q<\infty$ for every $q\ge1$;
		\item $T_S(V)$ is degreewise finite-dimensional over $k$.
	\end{enumerate}
\end{prop}

\begin{proof}
	Let $I=\ker(M\to S)$ be the augmentation ideal, with underlying graded module $V\otimes_S M$. The submodules $V(n)\otimes_SM$ form an exhaustive filtration of $I$ by dg submodules, and $d(V(n))\subseteq V(n-1)\otimes_SM$. As $S$ is semisimple, each subquotient is a direct summand of a direct sum of shifts of $M$, so $I$ is h-projective. Applying $\RHom_M(-,S)$ to the triangle $I\to M\to S\to$ and using $d(V)\subseteq I^2$, we obtain
	\[
		H^m\RHom_M(S,S)\simeq
		\begin{cases}
			S\oplus\Hom_{S}(V^1,S),&m=0,\\
			\Hom_{S}(V^{1-m},S),&m<0,\\
			0,&m>0,
		\end{cases}
	\]
	as graded $k$-vector spaces, where $\Hom_S$ denotes homomorphisms of right $S$-modules. Since $S$ is finite-dimensional semisimple, $\Hom_{S}(W,S)$ is finite-dimensional if and only if $W$ is finite-dimensional. Thus \cref{thm:main} gives \textup{(1)}$\Leftrightarrow$\textup{(2)} and \textup{(1)}$\Leftrightarrow$\textup{(3)}.

	Since $V$ is concentrated in positive degrees, each graded component of $T_S(V)$ is a finite direct sum of tensor products
	\[
		V^{q_1}\otimes_S\cdots\otimes_S V^{q_m},
		\qquad q_i\ge1,\quad q_1+\cdots+q_m=q.
	\]
	Hence \textup{(3)} implies \textup{(4)}, while the converse follows from $V^q\subseteq T_S(V)^q$.
\end{proof}

\begin{rmk}\label{rmk:minimal-tensor-existence}
	Let $A$ be a locally finite positive dg algebra such that $S=H^0(A)$ is separable over $k$. Then $A$ is quasi-isomorphic to a minimal tensor model as in \cref{prop:augmented-minimal-model}. We only sketch the argument. Since $S$-bimodules form a semisimple category, $H^*(A)$ carries a minimal $S$-relative $A_\infty$-structure, augmented by the projection $H^*(A)\to S$ \cite{Kad80}, \cite[Th\'eor\`eme~1.4.1.4 and Section~1.4.2]{Has} (see also \cite[Section~3]{SY19}). Its relative bar--cobar construction is a quasi-free tensor dg $S$-algebra with generators in positive degrees, quasi-isomorphic to $A$, and cancelling the linear part of its differential gives a minimal model.
\end{rmk}

\begin{rmk}\label{rmk:augmented-quiver-example}
	For $A=k[y]/(y^2)$ with $|y|=1$ and $d=0$, a minimal tensor model (with $S=k$) is
	\[
		M=k\langle x_1,x_2,x_3,\ldots\rangle,\qquad |x_n|=1,
	\]
	with differential
	\[
		d(x_1)=0,\qquad
		d(x_n)=\sum_{i=1}^{n-1}x_i x_{n-i}\quad(n\ge2).
	\]
	The dg algebra morphism $M\to A$ given by $x_1\mapsto y$ and $x_n\mapsto0$ for $n\ge2$ is a quasi-isomorphism. Indeed, up to sign conventions $M$ is the cobar construction of the coalgebra $C=\bigoplus_{n\ge0}kc_n$ with $\Delta(c_n)=\sum_{i+j=n}c_i\otimes c_j$, whose dual algebra is $k[\![x]\!]$, so $H^*(M)\simeq\Ext^*_{k[\![x]\!]}(k,k)\simeq k[y]/(y^2)$. Thus $M$ has infinitely many generators of degree one, in accordance with the failure of pvd-finiteness.
\end{rmk}

\begin{cor}\label{cor:pvd-derived-invariant}
The property of being pvd-finite is invariant under derived equivalence among locally finite positive dg algebras.
\end{cor}
\begin{proof}
Let $F\colon\D(A)\xrightarrow{\sim}\D(A')$ be a triangle equivalence and put $P=F^{-1}(A')$. Since equivalences preserve compact objects, $P\in\per(A)$. For $X\in\pvd(A)$,
\[
    H^i(FX)\simeq\Hom_{\D(A)}(P,\Sigma^iX).
\]
Because $P$ is perfect and $X$ has bounded finite-dimensional cohomology, the right-hand side is finite-dimensional and vanishes for $|i|\gg0$. Thus $F(\pvd(A))\subseteq\pvd(A')$; applying the same argument to a quasi-inverse gives equality. Hence Hom-finiteness of $\pvd$ is preserved, and the result follows from \cref{thm:main}(3).
\end{proof}

%%%%%%%%%%%%%%%%%%%%%%%%%%%%%%%%%%%%%%%%%%%%%%%%%%%%%%
\subsection{Morita characterization, length hearts and the ST correspondence}

\begin{thm}[Morita characterization]\label{thm:morita-characterization}
	Let $\T$ be a Hom-finite algebraic triangulated category. The following conditions are equivalent:
	\begin{itemize}
		\item[(1)] $\T$ has a length heart with an injective cogenerator;
		\item[(2)] $\T\simeq\per(A)$ for some pvd-finite positive dg algebra $A$;
		\item[(3)] $\T\simeq\pvd(B)^o$ for some locally finite non-positive dg algebra $B$.
	\end{itemize}
\end{thm}
\begin{proof}
	(1)$\Rightarrow$(2): Let $\H$ be a length heart with an injective cogenerator, let $L$ be the direct sum of representatives of the isomorphism classes of its simple objects, and put $A=\REnd(L)$. Since $\H$ is bounded, $L$ is a thick generator of $\T$ and $\T$ is idempotent complete \cite{LC07}. Hence Keller's dg Morita theorem \cite[Theorem~3.8(b)]{Ke06} gives $\T\simeq\per(A)$. The simple-minded property of $L$ makes $A$ positive, while Hom-finiteness of $\T$ makes $A$ locally finite. Under the Morita equivalence, $\H$ identifies with $\H_A$, so $A$ is pvd-finite by \cref{prop:pvd-basic}, since $\H$ has an injective cogenerator.

	(2)$\Rightarrow$(3): This follows from \cref{thm:KD}(2).

	(3)$\Rightarrow$(1): Suppose that $\T\simeq\pvd(B)^o$ for a locally finite non-positive dg algebra $B$. The standard bounded $t$-structure on $\pvd(B)$ has heart $\mod H^0(B)$. Thus $\T$ has the opposite heart $(\mod H^0(B))^o$, which is a length category with an injective cogenerator.
\end{proof}

\begin{rmk}\label{rmk:morita-opposite}
	The equivalent conditions in \cref{thm:morita-characterization} are invariant under passage to the opposite category. Indeed, if $\T\simeq\pvd(B)^o$, then
	\[
		\T^o\simeq\pvd(B)\simeq\pvd(B^o)^o
	\]
	via the $k$-dual, and $B^o$ is again a locally finite non-positive dg algebra.
\end{rmk}

\begin{thm}\label{thm:length-has-injcogen}
	Let $\T$ be a Hom-finite algebraic triangulated category. If some length heart of $\T$ has an injective cogenerator, then every length heart $\H$ of $\T$ is equivalent to $\mod\Lambda$ for some finite-dimensional $k$-algebra $\Lambda$. In particular, every length heart has a projective generator and an injective cogenerator.
\end{thm}
\begin{proof}
	By \cref{thm:morita-characterization}, we may assume that $\T=\per(A)$ for a pvd-finite positive dg algebra $A$.

	Let $\H$ be any length heart, let $L$ be the direct sum of representatives of the isomorphism classes of its simple objects, and put $C=\REnd_A(L)$. Since $L$ is a thick generator of $\per(A)$, dg Morita theory gives a derived equivalence between $A$ and $C$ identifying $\H$ with $\H_C$. Moreover, $C$ is locally finite and positive. By \cref{cor:pvd-derived-invariant}, $C$ is pvd-finite, so $\H_C$ has an injective cogenerator $I$ by \cref{prop:pvd-basic}. Since $\H_C$ is a Hom-finite length category with finitely many simple objects, $\Hom(-,I)$ gives an equivalence $\H_C^o\simeq\mod\Gamma$ for a finite-dimensional $k$-algebra $\Gamma$, and composing with the $k$-duality $(\mod\Gamma)^o\simeq\mod\Gamma^o$ gives $\H\simeq\H_C\simeq\mod\Lambda$ with $\Lambda=\Gamma^o$.
\end{proof}

\begin{prop}\label{thm:silting-simple-minded}
Let $A$ be a pvd-finite positive dg algebra. There is a commutative square of bijections:
\[
\begin{tikzcd}[
    column sep=5.0em,
    row sep=3.0em,
    cells={nodes={
        draw=black!65,
        rounded corners=2pt,
        fill=black!2,
        inner xsep=8pt,
        inner ysep=7pt,
        minimum width=15em
    }}
]
\begin{gathered}
    \text{simple-minded collections}\\
    \text{in }\per(A)\text{ up to isomorphism}
\end{gathered}
    \arrow[r,"\sim"] \arrow[d,"\sim"'sloped]
& \begin{gathered}
    \text{algebraic }t\text{-structures}\\
    \text{on }\per(A)
\end{gathered}
    \arrow[d,"\sim"sloped] \\
\begin{gathered}
    \text{silting objects in }\pvd(A)\\
    \text{up to equivalence}
\end{gathered}
    \arrow[r,"\sim"']
& \begin{gathered}
    \text{bounded weight structures}\\
    \text{on }\pvd(A)
\end{gathered}
\end{tikzcd}
\]
Under the bottom bijection, a silting object $M$ corresponds to the weight structure with coheart $\add M$. If $M$ corresponds to a heart $\H_M$, then
\[
    \H_M=\{X\in\per(A)\mid\Hom_{\D(A)}(X,\Sigma^{i}M)=0\text{ for }i\ne0\}.
\]
\end{prop}
\begin{proof}
Put $B=A^{!,o}$. By \cref{thm:KD},
\[
    \per(A)\simeq\pvd(B)^o,\qquad
    \pvd(A)\simeq\per(B)^o.
\]
Since $B$ is a locally finite non-positive dg algebra, the silting--$t$-structure (ST) correspondence of \cite[Theorem~3.7]{F23} gives the corresponding four structures on $\pvd(B)$ and $\per(B)$. Transporting it through these equivalences gives the asserted bijections, the coheart description, and the orthogonality formula for $\H_M$. For finite-dimensional algebras, such correspondences are due to \cite{KY14} (see also \cite{KV88}); for non-positive dg algebras see also \cite{SY19,Z23}, and \cite{AMY} for the framework of ST-pairs.
\end{proof}

%%%%%%%%%%%%%%%%%%%%%%%%%%%%%%%%%%%%%%%%%%%%%%%%%%%%%%%%%%%%%%%%%%%%%%%%%%%

\subsection{Double duality and adjunction}

Let $C$ be either a locally finite non-positive dg algebra or a pvd-finite positive dg algebra. Set
\[
	T_C:=\Phi_C(C)\in\Ddg(C^{!,o}),
	\qquad
	C^{!!}:=\End_{\Ddg(C^{!,o})}(T_C)^o.
\]
By \cref{thm:KD}, we have $T_C\simeq S_{C^{!,o}}$, so $C^{!!}\simeq(C^{!,o})^{!,o}$, which is quasi-isomorphic to $(C^!)^!$ by \cref{rmk:opposite-convention}. The dg functor $\Phi_C$ induces a morphism of dg algebras
\[
	\eta_C\colon
	C\cong\End_{\Cdg(C)}(C)
	\longrightarrow\End_{\Ddg(C)}(C)
	\xrightarrow{\Phi_C}
	\End_{\Ddg(C^{!,o})}(T_C)^o=C^{!!}.
\]

\begin{thm}\label{thm:Koszul-double-dual}
	The morphism $\eta_C\colon C\to C^{!!}$ is a quasi-isomorphism.
\end{thm}
\begin{proof}
	Since $C$ is h-projective as a right dg $C$-module, the morphism $\End_{\Cdg(C)}(C)\to\End_{\Ddg(C)}(C)$ is a quasi-isomorphism. The morphism induced by $\Phi_C$ is also a quasi-isomorphism, since $\Phi_C$ is fully faithful on $\per(C)$ by \cref{thm:KD}.
\end{proof}

\begin{prop}\label{thm:KD-adjunction}
	Let
	\[
		\widehat\Phi_C:=\RHom_{C^{!,o}}(-,T_C)
		\colon\D(C^{!,o})^o\longrightarrow\D(C^{!!}),
		\qquad
		F_C:=\operatorname{Res}_{\eta_C}
		\colon\D(C^{!!})\longrightarrow\D(C).
	\]
	Then $\Phi_C$ is left adjoint to $F_C\widehat\Phi_C$.
\end{prop}
\begin{proof}
	Choose an h-projective model $P_C$ of $S_C$. Then $\End_C(P_C)^o$ is a dg model of $C^{!,o}$, and under this identification $T_C=\Phi_C(C)$ corresponds to $P_C$ regarded as a right dg $\End_C(P_C)^o$-module; we use these models. Under these identifications, $\eta_C$ corresponds to the morphism $C\to\End_{C^{!,o}}(P_C)^o$ given by the right $C$-action on $P_C$, so that, for h-projective $N$, the dg $C$-module $\Hom_{C^{!,o}}(N,P_C)$ represents $F_C\widehat\Phi_C(N)$. The commuting right actions of $C$ and $C^{!,o}$ on $P_C$ give the dg Hom-currying isomorphism
	\[
		\Hom_{C^{!,o}}(N,\Hom_C(M,P_C))
		\simeq
		\Hom_C(M,\Hom_{C^{!,o}}(N,P_C))
	\]
	for h-projective modules $M$ and $N$. Passing to derived categories yields natural isomorphisms
	\[
		\RHom_{C^{!,o}}(N,\Phi_C M)
		\simeq
		\RHom_C(M,F_C\widehat\Phi_CN),
	\]
	and hence the asserted adjunction.
\end{proof}

\begin{rmk}\label{rmk:augmented-bar-cobar}
	Double Koszul duality in augmented settings is closely related to derived completion. In particular, Booth \cite[Theorem~4.2.8]{Boo22} proves double duality for connective (i.e.\ non-positive) augmented cohomologically locally finite dg algebras with Artinian local zeroth cohomology; see also \cite[Section~8.1]{Boo22} for the multi-pointed setting, where $k$ is replaced by $k^n$.

	Our construction does not require a chosen augmentation. Instead, it uses the distinguished module $S_C$, and double duality follows directly from the full faithfulness of the Koszul dual functor.
\end{rmk}

\begin{cor}\label{cor:KD-on-alg}
	Koszul duality induces a bijection
	\[
	\left\{
	\begin{array}{c}
	\text{locally finite non-positive}\\
	\text{dg algebras}
	\end{array}
	\right\}\big/\simeq
	\quad\overset{\mathrm{KD}}{\longleftrightarrow}\quad
	\left\{
	\begin{array}{c}
	\text{pvd-finite positive}\\
	\text{dg algebras}
	\end{array}
	\right\}\big/\simeq,
	\]
	where $\simeq$ denotes quasi-isomorphism. The correspondence in either direction is given by $C\mapsto C^{!,o}$.
\end{cor}
\begin{proof}
	Derived endomorphism dg algebras of distinguished modules are invariant under quasi-isomorphisms of dg algebras. Thus the correspondence is well defined on quasi-isomorphism classes, and it is involutive, since $(C^{!,o})^{!,o}\simeq C^{!!}\simeq C$ by \cref{thm:KD,thm:Koszul-double-dual}.
\end{proof}

%%%%%%%%%%%%%%%%%%%%%%%%%%%%%%%%%%%%%%%%%%%%%%%%%%%%%%%%%%%%%%%%%%%%%%%%%%%%%%%%%%%%%%%%%%%%%%%%%%%%%%%%%%%%%%%%%%%%%%%%%%%%%%%%%%%%
\section{Co-approximability and functorial finiteness}\label{section:ff}

We relate pvd-finiteness to co-approximability in the sense of \cite{Rah25} and prove two representability theorems for finite cohomological functors. We then study approximation properties of the canonical heart, characterize functorially finite bounded hearts, and recover Smal\o{}'s symmetry.

\subsection{Co-approximability and representability}

We use the notation of \cite[Notation~2.2]{Rah25}: for a set of objects $\mathcal S$ and integers $a\le b$, let $\operatorname{Coprod}_1(\mathcal S[a,b])$ be the closure of $\{\Sigma^{-i}U\mid U\in\mathcal S,\ a\le i\le b\}$ under coproducts, put $\operatorname{Coprod}_{n+1}(\mathcal S[a,b]):=\operatorname{Coprod}_1(\mathcal S[a,b])\ast\operatorname{Coprod}_n(\mathcal S[a,b])$, and let $\langle\mathcal S\rangle_n^{[a,b]}$ (denoted $\overline{\langle\mathcal S\rangle}{}_n^{[a,b]}$ in \cite{Rah25}) be the closure of $\operatorname{Coprod}_n(\mathcal S[a,b])$ under direct summands. Following \cite[Definition~7.8]{Rah25}, a triangulated category $\T$ with coproducts is \emph{co-approximable} if it admits a compact generator $G$, a weight structure $(\T_{\ge0},\T_{\le0})$ and an integer $N>0$ such that $\Sigma^{-N}G\in\T_{\ge0}$, $\Hom_{\T}(\Sigma^N G,\T_{\ge0})=0$ and $\T_{\ge0}\subseteq\langle G\rangle_N^{[-N,N]}\ast\T_{\ge1}$; the last condition is condition $(3)'$ of \cite[Definition~7.8]{Rah25}, which implies condition $(3)$ there. For a positive dg algebra we use the canonical weight structure.

\begin{thm}\label{thm:co-approximable}
Let $A$ be a pvd-finite positive dg algebra. Then $\D(A)$ is co-approximable with respect to the canonical weight structure. More precisely, there exists $n\ge1$ such that
\begin{equation}\label{eq:coapprox-unshifted}
	\D(A)_{\ge0}\subseteq\langle A\rangle_n^{[0,0]}\ast\D(A)_{\ge1}.
\end{equation}
\end{thm}
\begin{proof}
Write $\X^{\ast n}:=\X\ast\cdots\ast\X$ ($n$ factors), so that $(\Add A)^{\ast n}\subseteq\langle A\rangle_n^{[0,0]}$. Since $A$ is pvd-finite, there is $n\ge1$ with $\add S_A\subseteq(\add A)^{\ast n}\ast\D(A)_{\ge1}$ (\cref{dfn:pvd-finite,prop:positive-peeling}). By the argument of \cref{rmk:SA-unique}, every object of $\D(A)_0$ is a coproduct of objects of $\add S_A$, so $\D(A)_0\subseteq(\Add A)^{\ast n}\ast\D(A)_{\ge1}$.

Moreover, $\D(A)_{\ge2}\ast\Add A\subseteq\Add A\ast\D(A)_{\ge2}$ since $\Hom(\Add A,\Sigma\D(A)_{\ge2})=0$, and $\Sigma^{-1}\Add A\ast\Add A\subseteq\Add A\ast\Sigma^{-1}\Add A$ since, $H^0(A)$ being semisimple, every morphism in $\Add A$ is, up to isomorphism, the direct sum of an isomorphism and a zero morphism. Moving the factors $\D(A)_{\ge2}$ and $\Sigma^{-1}\Add A$ to the right of $(\Add A)^{\ast n}$, we obtain
\begin{align*}
	\D(A)_{\ge0}&=\D(A)_{\ge2}\ast\Sigma^{-1}\D(A)_0\ast\D(A)_0\\
	&\subseteq\D(A)_{\ge2}\ast(\Sigma^{-1}\Add A)^{\ast n}\ast\D(A)_{\ge2}\ast(\Add A)^{\ast n}\ast\D(A)_{\ge1}\\
	&\subseteq(\Add A)^{\ast n}\ast\D(A)_{\ge1},
\end{align*}
which proves \eqref{eq:coapprox-unshifted}. The remaining axioms hold for $N\ge n$, since $\Sigma^{-N}A\in\D(A)_{\ge0}$, $\Hom(\Sigma^NA,\D(A)_{\ge0})=0$ and $\langle A\rangle_n^{[0,0]}\subseteq\langle A\rangle_N^{[-N,N]}$.
\end{proof}

For a pvd-finite positive dg algebra $A$, the canonical weight structure lies in the preferred equivalence class by \cref{thm:co-approximable} and \cite[Proposition~7.11]{Rah25}, and the subcategories of \cite[Definition~3.17, Definition~7.6, and Remark~7.13]{Rah25} are $\D(A)^+_c=\bigcap_{m\in\ZZ}\bigl(\per(A)\ast\D(A)_{\ge m}\bigr)$, $\D(A)^-=\bigcup_{m\in\ZZ}\D(A)_{\le m}$ and $\D(A)^b_c=\D(A)^+_c\cap\D(A)^-$.

\begin{prop}\label{prop:metric-closure}
Let $A$ be a pvd-finite positive dg algebra. Then $\D(A)^+_c=\Dpfd(A)$ and $\D(A)^b_c=\pvd(A)$.
\end{prop}
\begin{proof}
If $X\in\D(A)^+_c$, taking $m=0$ shows that $X$ is bounded below. For fixed $i$ and $m>i$, a triangle $P_m\to X\to D_m\to$ with $P_m\in\per(A)$ and $D_m\in\D(A)_{\ge m}$ gives $H^i(X)\simeq H^i(P_m)$, which is finite-dimensional. Conversely, for $X\in\Dpfd(A)$ and any $m$, the $t$-structure of \cref{thm:main}(6) shifted by $m$ gives such a triangle with $P_m\in\per(A)^{\le m-1}$. Hence $\D(A)^+_c=\Dpfd(A)$, and the second equality follows since $\D(A)^-=\D^-(A)$.
\end{proof}

For an essentially small $k$-linear triangulated category $\mathcal C$, let $\operatorname{Coh}_{\mathrm f}(\mathcal C)$ denote the category of cohomological functors $F\colon\mathcal C^o\to\mod k$ such that $F(\Sigma^iC)=0$ for $|i|\gg0$ for every $C\in\mathcal C$. We call such functors \emph{finite cohomological functors}. Thus $\operatorname{Coh}_{\mathrm f}(\mathcal C^o)$ consists of homological functors $\mathcal C\to\mod k$.

\begin{thm}[Restricted Yoneda representability]\label{thm:restricted-yoneda}
Let $A$ be a pvd-finite positive dg algebra. The restricted Yoneda functors induce equivalences
\begin{itemize}
	\item[(1)] $\displaystyle\pvd(A)\xrightarrow{\sim}\operatorname{Coh}_{\mathrm f}(\per(A)),\qquad X\longmapsto\Hom_{\D(A)}(-,X)|_{\per(A)}$;
	\item[(2)] $\displaystyle\per(A)^o\xrightarrow{\sim}\operatorname{Coh}_{\mathrm f}(\pvd(A)^o),\qquad P\longmapsto\Hom_{\D(A)}(P,-)|_{\pvd(A)}$.
\end{itemize}
\end{thm}
\begin{proof}
(1) By \cref{thm:co-approximable} and \cite[Theorem~7.10]{Rah25}, $\D(A)$ is $\mathcal G$-approximable for the finite generating sequence $\mathcal G^m=\{\Sigma^mA\}$, hence $\mathcal G$-quasiapproximable by \cite[Proposition~5.3]{Rah25}. Thus \cite[Theorem~6.24]{Rah25} applies with the commutative Noetherian base ring $R=k$; the required finite generation of $\Hom(\Sigma^n A,\Sigma^m A)=H^{m-n}(A)$ is local finiteness of $A$. This theorem shows that the restricted Yoneda functor is fully faithful on $\D(A)^b_c$ with essential image the $\mathcal G$-finite cohomological functors on $\per(A)$, which for this generating sequence are exactly the objects of $\operatorname{Coh}_{\mathrm f}(\per(A))$ \cite[Example~6.7(2)]{Rah25}. Since $\D(A)^b_c=\pvd(A)$ by \cref{prop:metric-closure}, the claim follows.

(2) Put $B=A^{!,o}$. By \cref{thm:KD}(2), Koszul duality identifies $\per(A)^o$ with $\pvd(B)$ and $\pvd(A)^o$ with $\per(B)$, compatibly with the Hom-pairings since it is fully faithful on $\Dpfd(A)$. Thus the asserted functor becomes $\pvd(B)\to\operatorname{Coh}_{\mathrm f}(\per(B))$, $Y\mapsto\Hom_{\D(B)}(-,Y)|_{\per(B)}$. Since $\Hom_{\D(B)}(B,\Sigma^iB)=H^i(B)$ vanishes for $i>0$ and is finite-dimensional for all $i$, $\D(B)$ is approximable with its standard $t$-structure in the preferred equivalence class \cite[Remark~5.3]{Nee26}, and, as in \cite[Example~5.1]{Nee26}, local finiteness of $B$ gives $\D(B)^-_c=\Dmfd(B)$ (attach finitely many shifted copies of $B$ degree by degree from the top), so $\D(B)^b_c=\pvd(B)$. Hence the functor is an equivalence by \cite[Theorem~1.4(i)]{Nee26}; finiteness of cohomological functors is tested there on $B$, which agrees with ours since $\per(B)=\thick(B)$.
\end{proof}

\subsection{The canonical heart}

\begin{thm}\label{thm:pvd-finite-heart}
Let $A$ be a locally finite positive dg algebra. The following conditions are equivalent:
\begin{itemize}
	\item[(1)] $A$ is pvd-finite;
	\item[(2)] $\H_A$ is contravariantly finite in $\Dpfd(A)$;
	\item[(3)] $\H_A$ is contravariantly finite in $\per(A)$.
\end{itemize}
\end{thm}
\begin{proof}
Assume (1) and let $X\in\Dpfd(A)$. By weight orthogonality, a weight truncation $\sigma_{\ge0}X\to X$ is a right $\Dpfd(A)_{\ge0}$-approximation, and truncation for the $t$-structure of \cref{thm:main}(6) gives a right $\H_A$-approximation $H_X\to\sigma_{\ge0}X$; their composite is a right $\H_A$-approximation of $X$, proving (2). The implication (2)$\Rightarrow$(3) is immediate. Assume (3). By \cref{lem:heart-approx}(1), $\H_A$ has enough injectives; since it is a length category with finitely many simple objects, it has an injective cogenerator, and (1) follows from \cref{thm:main}(1).
\end{proof}

\begin{thm}\label{thm:smalo-type}
Let $A$ be a locally finite positive dg algebra. The following conditions are equivalent:
\begin{itemize}
	\item[(1)] $\pvd(A)\subseteq\per(A)$;
	\item[(2)] $\per(A)^{\ge0}$ is contravariantly finite in $\per(A)$;
	\item[(2)$'$] $\per(A)^{\le0}$ is covariantly finite in $\per(A)$;
	\item[(3)] $A^!$ is proper.
\end{itemize}
Here $(\per(A)^{\le0},\per(A)^{\ge0})$ is the canonical bounded $t$-structure with heart $\H_A$.
\end{thm}
\begin{proof}
Assume (1). For $X\in\per(A)$, a weight truncation $\sigma_{\ge0}X\to X\to\sigma_{<0}X\to$ has $\sigma_{<0}X\in\pvd(A)\subseteq\per(A)$, hence $\sigma_{\ge0}X\in\per(A)^{\ge0}$ by \eqref{eq:positive-adjacency}, and weight orthogonality makes $\sigma_{\ge0}X\to X$ a right $\per(A)^{\ge0}$-approximation. Conversely, assume (2), and complete a minimal right $\Sigma^{-1}\per(A)^{\ge0}$-approximation $\Sigma^{-1}P\to A$ to a triangle $\Sigma^{-1}P\to A\to S\to$. By \cref{fct:wakamatsu}, $H^i(S)=0$ for every $i>0$, while $S\in\per(A)^{\ge0}\subseteq\D(A)_{\ge0}$ since $\per(A)^{\ge0}$ is closed under extensions. Thus $S$ has finite-dimensional cohomology concentrated in degree zero. Choose $P'\in\add A$ and a morphism $P'\to S$ inducing an isomorphism on $H^0$; then the argument of \cref{rmk:SA-unique} gives $S\simeq\sigma_{\le0}P'\in\add S_A$. As $H^{-1}(P)=0$, the map $H^0(A)\to H^0(S)$ is injective, so every simple summand of $S_A$ occurs in $S$, and $\pvd(A)=\thick(S_A)=\thick(S)\subseteq\per(A)$. Thus (1)$\Leftrightarrow$(2).

Both (1) and (3) imply that $A$ is pvd-finite (\cref{thm:main}), and then, by \cref{thm:KD}(2), (1) is equivalent to the properness of $A^{!,o}$, i.e.\ to (3). Since $A^{o,!}\simeq A^{!,o}$ (\cref{rmk:opposite-convention}), condition (3) is symmetric in $A$ and $A^o$; applying (2)$\Leftrightarrow$(3) to $A^o$ and using the duality $\RHom_A(-,A)\colon\per(A)^o\simeq\per(A^o)$, which sends $\H_A$ to $\H_{A^o}$ and hence $\per(A)^{\le0}$ to $\per(A^o)^{\ge0}$, gives (2)$'$.
\end{proof}

The equivalence of (2) and (2)$'$ is a dg analogue of Smal\o{}'s symmetry \cite{S84}.

\subsection{Arbitrary bounded hearts}

We now turn to arbitrary bounded hearts. The key observation is the following general fact about abelian categories.

\begin{prop}\label{prop:enough-injectives}
Let $\H$ be a Hom-finite $k$-linear abelian category with enough injectives such that $K_0(\H)$ has finite rank. Then $\H$ has an injective cogenerator, and $\H$ is a length category with finitely many simple objects.
\end{prop}
\begin{proof}
For an injective object $J$ of $\H$, the functor $\Hom_{\H}(-,J)$ is exact, so $\lambda_J([X]):=\dim_k\Hom_{\H}(X,J)$ defines a homomorphism $\lambda_J\colon K_0(\H)\to\ZZ$. As $K_0(\H)$ has finite rank, there are injective objects $J_1,\ldots,J_m$ such that every $\lambda_J$ is a $\QQ$-linear combination of $\lambda_{J_1},\ldots,\lambda_{J_m}$. Put $E:=J_1\oplus\cdots\oplus J_m$. If $\Hom_{\H}(X,E)=0$, then $\lambda_J([X])=0$ for every injective object $J$, and taking a monomorphism $X\hookrightarrow J$ shows $X=0$.

Hence $\lambda_E([X])>0$ for every nonzero $X\in\H$, so every $X$ has length at most $\lambda_E([X])$. Next, let $f_1,\ldots,f_d$ be a basis of $\Hom_{\H}(X,E)$ and let $K$ be the kernel of $(f_1,\ldots,f_d)\colon X\to E^d$. Every morphism $K\to E$ extends to $X$, hence is the restriction of a linear combination of the $f_i$, which vanishes; thus $\Hom_{\H}(K,E)=0$ and $K=0$. Therefore $E$ is an injective cogenerator. Finally, every simple object is a subobject of $E$, which has finite length, so there are only finitely many simple objects up to isomorphism.
\end{proof}

\begin{thm}[Functorial finiteness of bounded hearts]\label{thm:ff-heart-characterization}
Let $\T$ be a Hom-finite algebraic triangulated category satisfying the equivalent conditions of \cref{thm:morita-characterization}, and let $\H$ be a bounded heart of $\T$. The following conditions are equivalent:
\begin{itemize}
	\item[(1)] $\H$ is a length heart;
	\item[(2)] $\H$ is contravariantly finite in $\T$;
	\item[(3)] $\H$ has enough injectives;
	\item[(4)] $\H$ has an injective cogenerator;
	\item[(5)] $\H$ is covariantly finite in $\T$;
	\item[(6)] $\H$ has enough projectives;
	\item[(7)] $\H$ has a projective generator;
	\item[(8)] $\H$ is functorially finite in $\T$.
\end{itemize}
If these conditions hold, then $\H\simeq\mod\Lambda$ for some finite-dimensional $k$-algebra $\Lambda$.
\end{thm}
\begin{proof}
By \cref{thm:morita-characterization}, we may assume that $\T=\per(A)$ for a pvd-finite positive dg algebra $A$. Since $\H$ and $\H_A$ are bounded hearts of $\T$, we have $K_0(\H)\simeq K_0(\T)\simeq K_0(\H_A)$, which has finite rank. Then (2)$\Rightarrow$(3) is \cref{lem:heart-approx}(1), (3)$\Rightarrow$(1) and (3)$\Rightarrow$(4) are \cref{prop:enough-injectives}, (4)$\Rightarrow$(3) is clear, and (1)$\Rightarrow$(4) and the final assertion follow from \cref{thm:length-has-injcogen}.

(1)$\Rightarrow$(2): Let $L$ be the direct sum of representatives of the isomorphism classes of simple objects of $\H$. Keller's dg Morita theorem identifies $\per(A)$ with $\per(A_L)$ for the locally finite positive dg algebra $A_L:=\REnd(L)$, and $\H$ with $\H_{A_L}$. Since $A_L$ is derived equivalent to $A$, it is pvd-finite by \cref{cor:pvd-derived-invariant}, and \cref{thm:pvd-finite-heart} shows that $\H$ is contravariantly finite.

Thus (1)--(4) are equivalent. The opposite category $\T^o$ satisfies the same hypotheses by \cref{rmk:morita-opposite}, so applying this to the bounded heart $\H^o$ of $\T^o$ gives the equivalence of (1), (5), (6) and (7). Finally, (8) is the conjunction of (2) and (5).
\end{proof}

\begin{cor}[Smal\o{}'s symmetry for module categories]\label{cor:smalo-module-symmetry}
Let $\Lambda$ be a finite-dimensional $k$-algebra and let $(\mathcal T,\mathcal F)$ be a torsion pair in $\mod\Lambda$. Then $\mathcal T$ is functorially finite in $\mod\Lambda$ if and only if $\mathcal F$ is functorially finite in $\mod\Lambda$.
\end{cor}
\begin{proof}
Since $\mod\Lambda$ is a length heart of $\pvd(\Lambda)=\D^{\mathsf b}(\mod\Lambda)$ with an injective cogenerator, \cref{thm:ff-heart-characterization} applies to $\pvd(\Lambda)$ by \cref{thm:morita-characterization}; in particular, $\mod\Lambda$ is functorially finite in $\pvd(\Lambda)$. As $\Lambda$ is proper, the standard aisle $\pvd(\Lambda)^{\le0}$ is covariantly finite in $\pvd(\Lambda)$ by \cref{thm:KD}(1) and \cref{thm:smalo-type}.

Assume that $\mathcal T$ is functorially finite in $\mod\Lambda$. Then $\mathcal T$ is covariantly finite in $\pvd(\Lambda)$, and so is $\pvd(\Lambda)^{<0}\ast\mathcal T$ by \cite[Theorem~1.3]{C09}. This is the aisle of a bounded $t$-structure, the HRS tilt \cite{H96}, whose heart $\Sigma\mathcal F\ast\mathcal T$ is therefore covariantly finite by \cref{lem:heart-approx}(2), hence contravariantly finite by \cref{thm:ff-heart-characterization}, and so is its torsion class $\Sigma\mathcal F$. Thus $\mathcal F$ is functorially finite in $\mod\Lambda$. The converse follows by applying this to the torsion pair $(\DD\mathcal F,\DD\mathcal T)$ in $\mod\Lambda^o$, where $\DD=\Hom_k(-,k)$.
\end{proof}

\section*{Acknowledgements}
The author thanks Hiroyuki Nakaoka for helpful discussions and feedback on earlier drafts, Dong Yang for invaluable comments and for suggesting several of the finiteness characterizations used in this paper, and Osamu Iyama for useful comments and support. The author is also grateful to Yuki Imamura, Nao Mochizuki, Shunya Saito and Ryu Tomonaga for reading this paper in a seminar and for their valuable feedback.

\section*{Use of AI}
The author used Claude (Anthropic), a large language model, in preparing this revision: to propose and draft revisions of the exposition and of some proofs, to check proofs, cross-references and bibliographic data, and for copy-editing. The author checked all mathematical content and takes full responsibility for the paper.

\bibliographystyle{alpha}
\bibliography{my}

\Addresses

\end{document}

%% file: pre.tex
\usepackage[margin=30truemm]{geometry}
\usepackage[T1]{fontenc}
\usepackage{amsmath,amsthm,amssymb}
\usepackage{lmodern}
\usepackage{microtype}
\usepackage{graphicx}
\usepackage{tikz-cd}
\usepackage{aliascnt}
\usepackage{enumitem}
\usepackage{xcolor}
\usepackage[numbers]{natbib}

\renewcommand{\emph}[1]{\textcolor{blue}{\textit{#1}}}
\definecolor{CoralPink}{HTML}{C94B55}
\definecolor{DarkOrange}{HTML}{B85C00}
\usepackage{hyperref}
\hypersetup{
    colorlinks=true,
    linkcolor=CoralPink, % internal references and table of contents
    filecolor=blue,
    urlcolor=DarkOrange,
    citecolor=DarkOrange
}
\usepackage{cleveref}

\makeatletter
\newcommand\ReDeclareMathOperator[2]{%
    \begingroup \escapechar\m@ne\xdef\@gtempa{{\string#1}}\endgroup
    \expandafter\@ifundefined\@gtempa
        {\@latex@error{Command \string#1 undefined}\@ehc}
        \relax
    \let\@ifdefinable\@rc@ifdefinable
    \DeclareMathOperator#1{#2}}
\makeatother

\crefname{thm}{Theorem}{Theorems}
\crefname{dfn}{Definition}{Definitions}
\crefname{prop}{Proposition}{Propositions}
\crefname{lem}{Lemma}{Lemmas}
\crefname{cor}{Corollary}{Corollaries}
\crefname{fct}{Fact}{Facts}
\crefname{rmk}{Remark}{Remarks}
\crefname{eg}{Example}{Examples}
\crefname{figure}{Figure}{Figures}
\crefname{table}{Table}{Tables}
\crefname{section}{Section}{Sections}
\crefname{subsection}{Subsection}{Subsections}
\crefname{equation}{Equation}{Equations}
\crefname{main}{Theorem}{Theorems}

\theoremstyle{definition}
\newtheorem{thm}{Theorem}[section]
\newtheorem{main}{Theorem}

\newaliascnt{dfn}{thm}
\newtheorem{dfn}[dfn]{Definition}
\aliascntresetthe{dfn}
\newaliascnt{prop}{thm}
\newtheorem{prop}[prop]{Proposition}
\aliascntresetthe{prop}
\newaliascnt{lem}{thm}
\newtheorem{lem}[lem]{Lemma}
\aliascntresetthe{lem}
\newaliascnt{cor}{thm}
\newtheorem{cor}[cor]{Corollary}
\aliascntresetthe{cor}
\newaliascnt{fct}{thm}
\newtheorem{fct}[fct]{Fact}
\aliascntresetthe{fct}
\newaliascnt{rmk}{thm}
\newtheorem{rmk}[rmk]{Remark}
\aliascntresetthe{rmk}
\newaliascnt{eg}{thm}
\newtheorem{eg}[eg]{Example}
\aliascntresetthe{eg}
\numberwithin{equation}{section}

\newcommand{\A}{\mathcal{A}}

\newcommand{\C}{\mathcal C}
\newcommand{\D}{\mathcal{D}}

\renewcommand{\H}{\mathcal H}

\newcommand{\Q}{\mathcal Q}

\newcommand{\T}{\mathcal T}
\newcommand{\U}{\mathcal U}

\newcommand{\X}{\mathcal X}
\newcommand{\Y}{\mathcal Y}

\newcommand{\RR}{\mathbf{R}}

\newcommand{\DD}{\mathbb{D}}

\newcommand{\QQ}{\mathbb{Q}}
\newcommand{\ZZ}{\mathbb{Z}}

\newcommand{\xto}{\xrightarrow}

\newcommand{\To}{\Rightarrow}

\DeclareMathOperator{\Filt}{Filt}
\DeclareMathOperator{\add}{add}
\DeclareMathOperator{\Add}{Add}
\ReDeclareMathOperator{\top}{top}

\DeclareMathOperator{\Mod}{Mod}
\let\mod\relax
\DeclareMathOperator{\mod}{mod}

\DeclareMathOperator{\Hom}{Hom}

\DeclareMathOperator{\End}{End}

\DeclareMathOperator{\RHom}{\RR Hom}

\DeclareMathOperator{\REnd}{\RR End}

\DeclareMathOperator{\thick}{thick}

\DeclareMathOperator{\Ext}{Ext}

\DeclareMathOperator*{\holim}{holim}
\DeclareMathOperator*{\hocolim}{hocolim}
\DeclareMathOperator{\cone}{cone}

\DeclareMathOperator{\per}{per}

\newcommand{\Dpfd}{\D^+_{\mathrm{fd}}}
\newcommand{\Dmfd}{\D^-_{\mathrm{fd}}}

\DeclareMathOperator{\pvd}{pvd}

\DeclareMathOperator{\Cdg}{\C_{dg}}

\DeclareMathOperator{\Ddg}{\D_{dg}}